\documentclass[11pt]{amsart}
\input epsf
\usepackage{latexsym}
\usepackage[usenames,dvipsnames]{color}
\usepackage[colorlinks=true, linkcolor=blue, citecolor=BrickRed, urlcolor=RoyalBlue]{hyperref}
\usepackage{amssymb, amsmath,amsthm,amscd, amssymb,
graphicx, enumerate, verbatim, mathrsfs, color, url, bm, hyperref, parskip}
\usepackage[all]{xy}\usepackage[labelformat=empty]{caption}

\usepackage[left]{lineno}
\newcommand{\Q}{{\mathbb Q}}

\def\Int{\operatorname{Int}}

\numberwithin{equation}{section}
\newtheorem{cor}[equation]{Corollary}
\newtheorem{lem}[equation]{Lemma}
\newtheorem{prop}[equation]{Proposition}
\newtheorem{thm}[equation]{Theorem}

\newtheorem{Example}[equation]{Example}
\newenvironment{ex}{\begin{Example}\rm}{\end{Example}}
\newtheorem{remark}[equation]{Remark}
\newenvironment{rmk}{\begin{remark}\rm}{\end{remark}}
\def\co{\colon\thinspace}

\newcommand{\e}{\varepsilon}

\def\a{\alpha}
\def\de{\delta}
\def\G{\Gamma}

\def\L{\Lambda}
\def\b{\beta}

\def\d{\partial}

\def\r{\rho}

\def\s{\sigma}

\def\R{\mathbb{R}}
\def\I{\mathbb{I}}

\def\Z{\mathbb{Z}}

\def\Z{\mathbb{Z}}
\def\S1{\bf S^1}

\newcommand{\overbar}[1]{\mkern 1.5mu\overline{\mkern-1.5mu#1\mkern-1.5mu}\mkern 1.5mu}

\makeatletter

\def\equalsfill{$\m@th\mathord=\mkern-7mu
\cleaders\hbox{$\!\mathord=\!$}\hfill
\mkern-7mu\mathord=$}
\makeatother

\begin{document}
%LINE NUMBERS
%\linenumbers
%END LINE NUMBERS

\abovedisplayskip=6pt plus3pt minus3pt
\belowdisplayskip=6pt plus3pt minus3pt

\title[Hyperbolization and regular neighborhoods]
{\bf Hyperbolization and regular neighborhoods}

%57Q40 Regular neighborhoods 
%57N55 Microbundles and block bundles
%20F67 Hyperbolic groups and nonpositively curved groups 
%53C23 Global geometric and topological methods (à la Gromov); differential geometric analysis on metric spaces 
\thanks{This work was partially supported by the Simons Foundation grant 524838.}

\author{Igor Belegradek}
\address{Igor Belegradek\\ School of Mathematics\\ Georgia Tech\\ Atlanta, GA, USA 30332\vspace*{-0.05in}}
\email{ib@math.gatech.edu}
%\urladdr{www.math.gatech.edu/$\sim$ib}

\begin{abstract}
We show that the hyperbolization of polyhedra pulls back  
regular neighborhoods of {\scshape pl} submanifolds. Applying this
to the Riemannian
version of the hyperbolization due to Ontaneda gives 
open complete manifolds of pinched negative curvature
that are homotopy equivalent to closed smooth manifolds but contain no
smooth spines. We also find open complete negatively pinched manifolds
that are homotopy equivalent to closed non-smoothable manifolds.
\end{abstract}
\maketitle
%\tableofcontents
\thispagestyle{empty}

\vspace{-17pt}

\section{Introduction}

In~\cite{Gro-hypgr} Gromov introduced several procedures 
that turn a simplicial complex $K$ into a polyhedron with 
piecewise-Euclidean metric of nonpositive curvature. 
Roughly speaking, every simplex of $K$ is replaced with a 
{\em hyperbolized simplex} of the same dimension, which are then assembled 
in the combinatorial pattern given by $K$. 
The procedures were further elaborated and developed in~\cite{DavJan, ChaDav}, and
in particular, the latter paper introduced the {\em strict hyperbolization}
that converts a finite simplicial complex $K$ into
a finite piecewise-hyperbolic locally CAT($-1$) polyhedron $\bm{K}$ of the same dimension such that
the following holds:
\begin{itemize}
\item[(1)] The procedure is functorial, i.e., an inclusion $L\subset K$ 
of a subcomplex induces an isometric embedding $\bm i\co \bm L\to \bm K$ onto a locally convex subset. 
Hence given $p\in \bm L$, the map $\bm i_*$ embeds $\pi_1(\bm L, p)$ 
onto a quasiconvex subgroup of the hyperbolic group $\pi_1(\bm K, p)$.
\item[(2)] The link of a simplex $L$ in $K$ is
{\scshape pl} homeomorphic to the link of $\bm L$ in $\bm K$. 
Hence if $L$, $K$ are {\scshape pl} manifolds, then so are $\bm{L}$, $\bm{K}$,
and furthermore, $L$ is locally flat in $K$ if and only if $\bm L$ is locally flat in $\bm K$.
\item[(3)] There is a continuous map $h\co \bm{K}\to K$ that 
\begin{itemize}
\item[(3a)]
pulls back rational Pontryagin classes, 
\item[(3b)] 
restricts to a homeomorphism on the one-skeleton, and hence induces 
a bijection $\pi_0(\bm K)\to\pi_0(K)$.
\item[(3c)]
is surjective on homology with coefficients in any commutative ring $R$ with identity. 
Hence if $K$ is a closed $R$-oriented manifold, then so is $\bm{K}$, and 
$h$ is injective on cohomology with coefficients in $R$. 
\end{itemize}
\end{itemize}

In particular, if $K$ is a 
connected closed oriented manifold, then so is $\bm K$.
We add to the above list with the following.

\begin{thm}
\label{thm: pullback}
Let $K$ be a closed {\scshape pl} manifold, and $L$ be a boundaryless
{\em(}not necessarily locally flat\,{\rm )} {\scshape pl} submanifold of $K$. 
If $h\co \bm K\to K$ is a  
hyperbolization map, and $\bm L=h^{-1}(L)$, then $h$ pulls
the regular neighborhood of $L$ in $K$ back to the regular neighborhood of 
$\bm L$ in $\bm K$.
\end{thm}

Even though Theorem~\ref{thm: pullback} 
refers to the strict hyperbolization, it also holds 
for other hyperbolization procedures, see Section~\ref{sec: hyperb pulls back}.
The proof of Theorem~\ref{thm: pullback} hinges on the stratified
block transversality due to Stone~\cite{Sto72}, which reduces to the work of
Rourke and Sanderson~\cite{RS-II} when $\dim(K)-\dim(L)>2$. 

To make sense of Theorem~\ref{thm: pullback} let us review properties of
regular neighborhoods, which are also known as thickenings; the
two terms will be used interchangeably, see Section~\ref{sec: thick} for details. 
Basically, thickenings  and regular neighborhoods  
are related in the same way as vector bundles and normal bundles of smooth submanifolds.  
Recall that a regular neighborhood of a 
smooth codimension $q$ submanifold $M$ is a linear disk bundle,
and isomorphism classes of such bundles are classified by homotopy classes of maps
from $M$ to $BO_q$. A similar theory for (compact boundaryless) 
{\scshape pl} submanifolds $M$ was developed by Rourke and Sanderson 
for $q\ge 3$~\cite{RS-I}, and Cappell and Shaneson for $q=2$~\cite{CS76}, 
and a unified treatment for any $q\ge 1$ can be found in~\cite{Mc-cone}.
The only known thickenings 
of codimension $q=1$ are the \mbox{$I$-bundles}, and the yet 
unresolved Schoenflies conjecture predicts that there is no other
examples. 
We refer to $M$ as a {\em spine\,} of the thickening; more generally,
any closed manifold that is a deformation retract of a manifold 
is called a {\em spine}.
Thickenings can be pullbacked via continuous maps, and 
concordance classes of codimension $q$ thickenings of $M$  
bijectively correspond to the homotopy classes of maps of 
$M$ into a certain classifying space.  
Concordant thickenings of $M$ are {\scshape pl} homeomorphic rel $M$, 
and the converse is true if $q\ge 3$.

Ontaneda~\cite{Ont-ihes} established a Riemannian version of strict hyperbolization: 
for every $\e>0$ and a smoothly triangulated closed 
$k$-manifold $K$ with $k\ge 2$ there is
a hyperbolized $k$-simplex as in~\cite{ChaDav} such that the corresponding 
strict hyperbolization $\bm K$  admits a Riemannian metric $g_{\e, K}$
with sectional curvature within $[-1-\e, -1]$.
Here is a smooth version of Theorem~\ref{thm: pullback}.

\begin{cor} 
\label{cor: Ontaneda smooth}
If $L$ is a boundaryless smooth 
submanifold of closed smooth manifold $K$ of dimension $k\ge 2$ 
that is smoothly triangulated with $L$
as a subcomplex, then there is a hyperbolized 
$k$-simplex and a smooth structure on $\bm K$
such that 
\begin{itemize}
\item[\textup(a)] 
$\bm L$ is a smooth submanifold of $\bm K$
whose normal bundle is isomorphic as a linear disk bundle
to the pullback via $h$ of the normal bundle to $L$ in $K$,
\item[\textup(b)] $\bm K$ admits a Riemannian metric 
with sectional curvature within $[-1-\e, -1]$.
\end{itemize}
\end{cor}

For example, if $L$ is any smoothly embedded  $2$-sphere in an 
oriented $4$-manifold $K$ with normal 
Euler number $e$, then Corollary~\ref{cor: Ontaneda smooth} 
gives a genus $g$ hyperbolic 
surface $\bm L$ with the same normal Euler number $e$
in the negatively pinched closed
Riemannian $4$-manifold $\bm K$. 
Here $g$ depends on a smooth triangulation of
$K$ that contains $L$ as a subcomplex, and in particular, 
$g$ depends on $e$.

Here is a way to produce open complete 
negatively curved Riemannian manifolds.
If $L\to K$ is a {\scshape pl}-embedding of closed manifolds, 
then the inclusion $\bm L\to\bm K$ is $\pi_1$-injective, and hence
the subgroup $\pi_1(\bm L)$ corresponds to a covering space 
$\overbar{\bm K}$ of $\bm K$. Since both $\bm L$ and $\bm K$ 
are aspherical, $\overbar{\bm K}$ deformation retracts
onto a ({\scshape pl}-embedded) lift of $\bm L$.
If $K$ is smoothable, the pullback of Ontaneda's negatively pinched
Riemannian metric on $\bm K$ 
to $\overbar{\bm K}$ is {\em convex-cocompact}, i.e., 
$\overbar{\bm K}$ deformation retracts onto a 
compact codimension zero locally convex subset 
whose interior is diffeomorphic to
$\overbar{\bm K}$, see Remark~\ref{rmk: convex-cocompact}.
To avoid confusion 
we stress that $\bm L$ need not be locally convex 
in Ontaneda's Riemannian metric on $\bm K$, and in fact,
the submanifold $\bm L$ may not even be smoothable, which is
the main theme of this paper. 
By contrast, the strict hyperbolization ensures that
$\bm L$ is locally convex in the 
locally CAT($-1$) metric on $\bm K$.

Previously known open manifolds that are homotopy equivalent to closed manifolds
and admit complete negatively pinched metrics are as follows:
\begin{itemize}
\item The total space of any vector bundle over a closed negatively curved manifold admits 
a complete pinched negatively curved metric~\cite{And-vb}. 
\item In~\cite{GLT, Kui, Kap89, Bel-odd, GKL, AGG} one finds 
complete locally symmetric metrics of negative curvature that are interiors
of codimension $2$ thickenings of real hyperbolic manifolds of dimensions $\le 3$.
Some of these have smooth spines, while others have non-locally-flat 
{\scshape pl} spines.
\item If $Y$ is a totally geodesic submanifold in a symmetric space $X$ of negative curvature,
and $\Gamma$ is a discrete torsion-free isometry group of $X$ that stabilizes $Y$, 
then the nearest point projection $X\to Y$ is a $\Gamma$-equivariant vector bundle,
and hence $X/\Gamma$ is the total space of a vector bundle over $Y/\Gamma$. 
\end{itemize}

The following theorem exploits the differences between smoothable 
{\scshape pl} thickenings and linear disk bundles.
Here $\left[q/2\right]$ is the largest integer that is $\le q/2$.

\begin{thm}
\label{thm: non-smooth}
Let $l, q\in\Z$ such that 
either $\ l\ge 2=q\ $ or $\ l\ge 4\left[q/2\right]>0$.
Then for any $\e>0$ there is 
a closed Riemannian $(q+l)$-manifold $\bm K$ of sectional curvature 
in $[-1-\e, -1]$, 
a closed smooth locally CAT$(-1)$ $l$-manifold $\bm L$ 
that is {\scshape pl}-embedded into $\bm K$,
and a convex-cocompact covering space $\overbar{\bm K}$ of $\bm K$ 
such that\vspace{4pt} \newline
\textup{$\phantom{\quad}$(i)} $\overbar{\bm K}$ deformation retracts 
onto a {\scshape pl}-embedded lift of $\bm L\subset\bm K$,
\vspace{4pt} \newline
\textup{$\phantom{\quad}$(ii)} $\overbar{\bm K}$ has no deformation retraction
onto a smooth $l$-dimensional submanifold.\vspace{4pt}
Moreover, if $l\ge 4$, then
any finitely-sheeted cover of $\overbar{\bm K}$ deformation retracts
onto a {\scshape pl}-embedded closed $l$-manifold, but admits no 
deformation retraction a smoothly embedded closed $l$-manifold.
\end{thm}

The conclusion of Theorem~\ref{thm: non-smooth} for $l\ge 4$ 
hinges on properties of certain rational characteristic classes 
which live in cohomology of the classifying space of oriented 
codimension $q$ thickenings. 
The task is to construct thickenings whose characteristic classes
satisfy an identity that fails in $H^*(BSO_q;\Q)$.
The classes are preserved by the canonical homotopy equivalence 
between different spines, and also survive under $h$
(because it is injective on rational cohomology). Hence 
there is no smooth spine.

If $l$ is $2$ or $3$, the above-mentioned characteristic classes vanish, and instead we argue that there are thickenings of $L=S^2$ and $L=S^2\times S^1$
with nontrivial normal invariant, and the same is true for $\bm L$
because $h$ is surjective on homology. 
The canonical homotopy equivalence between between different spines
is homotopic to a {\scshape pl} homeomorphism (because $l\le 3$), and
therefore has trivial normal invariant. 
By a result of Cappell and Shaneson~\cite{CS76}
the two normal invariants have to agree, hence there is no smooth spine.

By Haefliger's metastable range embedding theorem~\cite{Hae61} any homotopy equivalence from a closed smooth $l$-manifold
to an open smooth $(l+q)$-manifold is homotopic to a smooth embedding 
if $l< 2q -2$. 
Thus the dimension bound in Theorem~\ref{thm: non-smooth} is sharp for odd $q$,
and close to being sharp for even $q$.

Finally, here is a
version of Theorem~\ref{thm: non-smooth} with non-smoothable $\bm L$. 
The proof does not use Theorem~\ref{thm: pullback}.

\begin{thm}
\label{thm: non-smoothable spine}
If $\e>0$, $k\in\Z$ and $l\in 4\Z$ with $k\ge 2l-1\ge 15$, then 
there is a $\pi_1$-injective {\scshape pl} embedding $\bm{L}\to \bm{K}$ of closed manifolds with \mbox{$\dim(\bm{L})=l$} and $\dim(\bm{K})=k$, and such that 
\begin{itemize}
\item
$\bm{L}$ is locally CAT$(-1)$, and not homotopy equivalent to a 
smooth $l$-manifold, 
\item
$\bm{K}$ has Riemannian metric of sectional curvature within $[-1-\e, -1]$,
\item 
$\bm K$ has 
a convex-cocompact covering space $\overbar{\bm K}$ 
that deformation retracts onto a 
{\scshape pl}-embedded lift of $\bm L$. 
\end{itemize}
\end{thm}

{\bf Acknowledgment} I am grateful to Colin Rourke for showing me in mid 1990s
how to compute the rational homotopy type of $BS\widetilde{PL}_q$, to John Klein
for refreshing my memory on the issue in~\cite{Kle-MO}, 
to Pedro Ontaneda for discussions on {\scshape pl}
transversality in relation to the hyperbolization,
which prompted me to write Appendix~\ref{app: transversality}, 
and to Oscar Randal-Williams for~\cite{RW-MO}.

{\bf Structure of the paper} 
Theorem~\ref{thm: pullback} and Corollary~\ref{cor: Ontaneda smooth}
are proved in Section~\ref{sec: hyperb pulls back} which is independent of
the rest of the paper. 
Section~\ref{sec: thick} gives
background on thickenings whose classifying spaces 
are discussed in Sections~\ref{sec: class block bundles} 
and~\ref{sec: classifying spaces of codimension 2 thickenings}.
Theorem~\ref{thm: non-smooth} for $q>2$
is proved in Section~\ref{sec: Smoothable thickenings}.
The case $q=2$ is treated in Section~\ref{sec: classifying spaces of codimension 2 thickenings} where we also establish 
Theorem~\ref{thm: non-smoothable spine}.
The three appendices contain some technical lemmas.

{\bf Notations and conventions} 
Let $\Z_{>0}$ and $\Z_{\ge 0}$ 
denote positive and nonnegative integers, respectively. 
Set $I^q=[-1,1]^q$, $I=[-1,1]=I^1$, and $\I=[0,1]$. 
Unless stated otherwise 
we work in the category of polyhedra and {\scshape pl} maps, and in particular,
{\em all manifolds, homeomorphisms, isotopies, and embeddings are {\scshape pl}};
see~\cite{Sta-tata, RS-book} for {\scshape pl} topology background.
A {\em submanifold\,} is a subcomplex homeomorphic to a manifold; 
thus a submanifold need not be locally flat.

\section{Background on thickenings} 
\label{sec: thick}
 
{\bf Notation} In this section $M$ is a closed connected $m$-manifold, $q\in\Z_{>0}$,
and $V$ is an open $(m+q)$-manifold.

If $M$ is a subpolyhedron of a compact $(m+q)$-manifold $W$
that collapses onto $M$~\cite[Chapter 3]{RS-book}, 
then $W$ is a {\em codimension $q$ thickening\,} of $M$. 
Two thickenings of 
$M$ are {\em equivalent\,} if they are homeomorphic rel $M$.
Any thickening $W$ of $M$ is equivalent to the star of $M$
in the second-derived subdivision in (any triangulation of) $W$,
and conversely, for any manifold that contains $M$ as a subpolyhedron
the star of $M$ in the second-derived subdivision is a thickening of $M$.

Thickenings with locally flat $M$ are precisely the
block bundles with fiber $I^q$~\cite[Corollary 4.6]{RS-I}.
By Zeeman's unknotting theorem $M$ is locally flat 
if $q\ge 3$~\cite{RS-book}. 

If a codimension $1$ thickening with 
non-locally flat $M$ exists, it would give a counterexample to the 
{\scshape pl} Schoenflies conjecture, 
which is true when $m\le 2$ and open for $m>2$.
Thus all known codimension one thickenings are block bundles.
Also every codimension one
thickening is topologically (rather than {\scshape pl}) homeomorphic to a linear $I$-bundle~\cite{Bro62, Mic64}. 

By contrast, thickenings of codimension $2$ often contain $M$ as a 
non-locally-flat submanifold, such as, e.g., a regular
neighborhood of embedded $2$-sphere in $S^4$
obtained as the suspension a {\scshape pl} knotted circle in $S^3$. 

If $q\le 2$, every block bundle with fiber $I^q$ has a structure of a 
linear disk bundle, which is unique up to an isomorphism~\cite[p.127]{Wal-book}. 
Thus block bundles and thickenings differ for $q=2$, and possibly for $q=1$.

Two thickenings of $M$ are {\em concordant\,} 
if they appear as boundaries of a thickening of $M\times \I$, 
see~\cite[p.172]{CS76} for details. 
Concordant thickenings are equivalent, i.e., homeomorphic rel $M$. 
Equivalent thickenings are concordant if $q\ge 3$~\cite[Corollary 1.8]{RS-I}.
If $q=2$, it is not easy to decide if two given non-concordant thickenings are
equivalent.
 
As we discuss later, the concordance classes of codimension $q$ thickenings
form a classifying space, which was constructed in~\cite{RS-I} for block bundles,
in~\cite{CS76} for codimension $2$ thickenings, and a unified construction 
for all $q\ge 1$ was given in~\cite{Mc-cone}. The homotopy type
of the classifying space is fairly well-understood, except when $q=1$ again
due to unresolved Schoenflies conjecture. 
The classifying space for block bundles 
of codimension $q\le 2$ is homotopy equivalent 
to $BO_q$.

Here is a common way of how thickenings arise.
Consider a homotopy equivalence $f\co M\to V$.
\begin{itemize}
\item
If $M$, $V$ are smooth and
$m\le 2q-3$, then $f$ is homotopic to a smooth embedding~\cite{Hae61},
whose normal disk bundle is a thickening. 
\item
If $q\ge 3$, then $f$ is homotopic to an embedding~\cite[Theorem 5.2.1]{DV-book}, 
and in particular, any regular neighborhood of $f(M)$ is a codimension 
$q$ thickening, which we call {\em the normal block bundle of $f$.}
\item
If $q=2$ and $V$ is the interior of a compact manifold with boundary, 
then~\cite[Theorem 6.1]{CS76} any homotopy equivalence $f\co M\to V$
is homotopic to an embedding if one of the following is true:
\begin{itemize}
\item
$m$ is odd and $f$ is {\em orientation-true} (i.e. $f$
preserves the first Stiefel-Whitney class of the tangent microbundles $TM$, $TW$), 
\item
$m$ is even, $m\neq 2$, and $M$ is simply connected.
\end{itemize}
\end{itemize}  
If $q=2$, the above restrictions cannot be dropped: 
There are so-called totally spineless manifolds $V$, which are
interiors of compact manifolds, 
such that no homotopy equivalence $M\to V$ is homotopic
to an embedding, where $M=T^2$~\cite{Mat75}, $M=S^2$~\cite{LevLid},
or any $M$ with even $m\ge 4$ and $H_1(M)\neq 0$~\cite{CapSha-symp}. 
If we do not assume that $V$ is the interior of a compact manifold, 
then more can be said: 
For every $m\ge 4$ there is an open smooth $(m+2)$-dimensional 
manifold $V$ 
that is homotopy equivalent to $S^2\times S^{m-2}$ but does not deformation retract to any compact subset, and there is a similar example with 
$m=2$ and $S^2\times S^{m-2}$ replaced by $S^2$~\cite{Ven}.

If $W$ is a thickening of $M$ of codimension $q\ge 3$,
then the inclusion $\d W\to W$ followed by a deformation retraction $W\to M$
is (homotopy equivalent over $M$ to)
a spherical fibration with fibers homotopy equivalent to $S^{q-1}$~\cite[Corollary 5.9]{RS-I}.
If this spherical fibration is orientable, its Thom class
generates $H^q(W,\d W)\cong\Z$, and we call the generator 
an {\em orientation} of $W$.

Define an {\em orientation} of a codimension $2$ thickening $W$ of $M$ as the orientation 
of the spherical fibration $\d W^\prime\to M$, 
where $W^\prime:=W\times I\supset M\times\{0\}=M$. 
The K\"unneth formula for $(W,\d W)\times (I, \d I)$ 
gives a natural isomorphism $H^2(W,\d W)\cong H^3(W^\prime,\d W^\prime)$,
%Dold, chapter VII proposition 7.6
hence this definition agrees with~\cite[p.172]{CS76}. 
Thus an orientation of a codimension $2$ thickening can also be thought
of as a generator in $H^2(W,\d W)$, cf.~\cite{Spi67}.

A thickening is {\em orientable\,} if it has an orientation.
It is easy to see that a thickening is orientable if and only if the inclusion $M\to W$
is orientation-true. 
An {\em equivalence of oriented thickenings\,} is a homeomorphism rel $M$
that preserves the orientation. 

Define the {\em Euler class\,} $e$ of an oriented codimension $q$
thickening $W$ of $M$ as the image of its orientation under the 
homomorphism $H^q(W,\d W)\to H^q(M)$ induced by the inclusion. 
If $q=2$, the definition appears in~\cite[p.178]{CS76}.
If $q\ge 3$, then $e$ is precisely the Euler class 
of the spherical fibration $\d W\to M$. 

In fact, one can define the Euler class $e(f)$ of any continuous
map $f\co M\to V$ of oriented manifolds as the image of 
the fundamental class $[M]$ under the composite
\begin{equation}
\label{diagram: euler class open}
H_n(M)\underset{f_*}{\to} H_n^{\mathrm{lf}}(V)\underset{PD}{\cong} H^q(V)\underset{f^*}{\to} H^q(M)
\end{equation}
where the middle map is the Poincar\'e duality isomorphism between the 
homology based on locally finite chains and the singular cohomology.
(More generally, the definition works if $f$ is orientation-true
and the homology groups in (\ref{diagram: euler class open}) have twisted
coefficients, which we avoid here). 

Here is a sketch on why the above definitions of the Euler class
agree when $f$ is the inclusion and $V$ is the interior of a thickening:
The Thom class can be thought of as the element in 
$H^q(V, V\setminus M)$ that is Poincar\'e dual to the fundamental 
class in $H_m(M)$, see~\cite[Chapter VIII, section 11]{Dol-book}, 
and this Poincar\'e duality isomorphism fits in a commutative
square as in~\cite[Theorem 5.9.3]{Bre-sheaf} whose other arrows are
the Poincar\'e duality isomorphism in 
(\ref{diagram: euler class open}) and 
the inclusion-induced maps $H_n(M)\to H_n^{\mathrm{lf}}(V)$ and 
$H^q(V, V\setminus M)\to H^q(V)$.

One can define the total rational Pontryagin class
of a continuous map $f\co M\to V$ via
$p(f)=p(\nu_M\oplus f^*\tau_V)$
by mimicking the Whitney sum formula, 
where $\nu_M$, $\tau_V$ are the stable normal and tangent microbundles of 
$M$, $V$, respectively. 
If $f$ is an embedding, then 
its stable normal microbundle $\nu_f$ is $\nu_M\oplus f^*\tau_V$~\cite{Mil-micro},
and $p(f)$ is the Pontryagin class of $\nu_f$. It then follows from~\cite{RW-MO} that
$p(f)=p(\nu_M)\cup f^*p(TV)$. 

If $f$ is an inclusion, $e(f)$ and $p(f)$ are called
the {\em normal Euler class\,} and the {\em normal rational Pontryagin class\,} of 
$M$ in $V$, respectively.

For the the remainder of this section we discuss manifolds 
that deformation retract to two different compact boundaryless 
submanifolds, which is the setting of Theorem~\ref{thm: non-smooth}. 

\begin{lem}
\label{lem: sphere bundles fhe}
If two thickenings $W$, $W^\prime$ of $M$, $M^\prime$, respectively, share 
the same interior $U$, then there is a fiber homotopy equivalence of 
the corresponding spherical fibrations that covers the composite 
of the inclusion $j^\prime\co M^\prime\to U$ followed by a 
deformation retraction $\pi\co U\to M$. 
\end{lem}
\begin{proof}
This a minor modification of the same assertion for vector bundles 
proved in~\cite[Proposition 4.1]{BKS11}. The only change is to replace
smooth tubular neighborhoods by regular neighborhoods, and the 
point is that $U$ is the union of a countable family of nested regular neighborhoods
of $M$ obtained by removing small open collars of $\d W$ in $W$,
and the same is true for $M^\prime$, $W^\prime$. 
\end{proof}

\begin{prop} 
\label{prop: engulf}
Let $M$, $M^\prime$ be closed $m$-manifolds embedded 
into an open manifold $V$ such that there are deformation retractions
$\pi\co V\to M$ and $V\to M^\prime$.
\begin{itemize}
\item[\textup{(1)}]
If $\dim(V)\ge m+3$,
then there is a fiber homotopy equivalence of spherical fibrations 
of the normal block bundles to $M$, $M^\prime$ that covers
$\pi\vert_{M^\prime}$.
\item[\textup{(2)}]
If $M$, $M^\prime$, $V$ are oriented and $\pi\vert_{M^\prime}\co M^\prime\to M$
orientation-preserving, then $\pi\vert_{M^\prime}$  
preserves the normal Euler classes to $M$ and $M^\prime$ in $V$. 
\item[\textup{(3)}]
If $\pi\vert_{M^\prime}$ maps $p(\nu_M)$ to $p(\nu_{M^\prime})$,
then $\pi\vert_{M^\prime}$ preserves the normal rational
Pontryagin classes to $M$ and $M^\prime$ in $V$.
\end{itemize}
\end{prop}
\begin{proof} (1) We can assume $\dim(V)\ge 5$. 
If $R$ is a regular neighborhood of $M$,
then the inclusion 
$R\to V$ is a homotopy equivalence. Since $m\ge 3$, 
Stallings engulfing~\cite[Theorem 3.1.3]{DV-book}
gives an ambient isotopy that moves $M$ to a submanifold
$L\subset\mathrm{Int}(R)$. 
If $W$ is a regular neighborhood of $L$ in $\mathrm{Int}(R)$,
then the inclusion $W\to\mathrm{Int}(R)$ is a homotopy equivalence, and hence
$R\setminus\mathrm{Int}(W)$ is an h-cobordism~\cite[p.213]{Sie-open-collar},
so $\mathrm{Int}(R)\setminus\mathrm{Int}(W)$ 
is homeomorphic to $[0,1)\times\d W$
rel $\{0\}\times\d W$. Hence $R$, $W$ have homeomorphic interiors. 
Lemma~\ref{lem: sphere bundles fhe} gives desired 
fiber homotopy equivalence of the
associated the sphere bundles. 

(2) Since $\pi$ is homotopic to the identity of $V$, and $\pi\vert_{M^\prime}$
is orientation-preserving, the images of the fundamental classes
$[M]$, $[M^\prime]$ in $H_n^{\mathrm{lf}}(V)$ are equal, and after taking the Poincare dual
and restricting to $M$, $M^\prime$ we get Euler classes
that are taken to each other by $\pi\vert_{M^\prime}$.

(3) Since $\pi$ is homotopic to the identity of $V$, 
the map $\pi\vert_{M^\prime}$ takes $\tau_V\vert_M$ to $\tau_V\vert_{M^\prime}$.
By assumption it also takes $p(\nu_M)$ to $p(\nu_{M^\prime})$,
and hence sends the normal rational Pontryagin class to $M$ to
the normal rational Pontryagin class to $M^\prime$. 
\end{proof}

\begin{rmk}
\label{rmk: borel}
We shall apply (3) when $M$, $M^\prime$ are closed locally CAT$(-1)$ manifolds, 
in which case any homotopy equivalence
preserves the rational Pontryagin class of the stable normal microbundle. 
(If $\dim(M)\ge 5$ this follows from
the solution of the Borel conjecture for CAT$(0)$ groups~\cite{BL}
and topological invariance of rational Pontryagin classes.
If $\dim(M)=4$, one can either apply the argument of the previous sentence
to the product of the homotopy equivalence $M\to M^\prime$
and the identity map of $S^1$, or more classically, note that
the only relevant class is $p_1$, and 
if $M$ is oriented, $p_1(\tau_M)$ is proportional to the signature
$\s(M)$ which is an oriented homotopy invariant, and 
the non-orientable case easily follows by passing to orientable
$2$-fold covers using that finitely-sheeted covering maps 
are injective on rational cohomology). 
\end{rmk}

\begin{rmk}
In Section~\ref{sec: class block bundles} we shall define for every odd 
$q\ge 3$ a characteristic class $o_q$ of oriented spherical fibrations. 
If in Proposition~\ref{prop: engulf}(1)
the fiber homotopy equivalence preserves orientation of the fiber,
then $\pi\vert_{M^\prime}$ preserves $o_q$.
\end{rmk}

If $W$ is a 
thickening of $M$ and $\pi\co\overbar{W}\to W$ is a finitely-sheeted
covering map, then $\overbar{W}$ is a thickening of $\overbar{M}:=p^{-1}(M)$,
e.g. because the $\pi$-preimage of 
the star of $M$ in a second-derived subdivision of $W$ is clearly
the star of $\overbar{M}$ in a second-derived subdivision of $\overbar{W}$.

\begin{prop}
\label{prop: finite cover}
The classes $e$, $p$, $o_q$ of $\overbar{W}$
are $\pi^*$-images of the corresponding classes 
$e$, $p$, $o_q$ of $W$.
\end{prop}
\begin{proof}
If $W$ is a block bundle over $M$, then the definition of a pullback 
in~\cite{RS-I} easily implies that the pullback $\pi^*W$ of $W$ via 
$p\co \overbar{M}\to M$
is a covering space of $W$ that extends $p$, and 
uniqueness of the covering space shows that $\overbar{W}$ and $\pi^*W$
are equivalent, which proves the claim for $q\ge 3$. 

The pullback of a codimension $q<3$ thickening is defined in
a less explicit way, see~\cite{CS76}, and here we merely observe
that $e$ and $p$ pullback via $\pi$. 
For the Euler class the claim follows by
observing that the Thom class in $H^q(W, W\setminus M)$
visibly maps to the Thom class in 
$H^q(\overbar{W}, \overbar{W}\setminus \overbar{M})$
because it is represented by a $q$-disk transverse to $M$
at a locally flat point. 
For the Pontryagin class 
the claim holds because covering maps
pullback tangent microbundles. 
\end{proof}

\section{Classifying spaces for block bundles and spherical fibrations}
\label{sec: class block bundles}

This section reviews some results 
in~\cite{Mil, RS-I, RS-II, RS-III, CS76, MadMil-book}.

Let $G_q$ be the space $G_q$ of homotopy self-equivalences of $S^{q-1}$. Then
$BG_q$ is the classifying space for fibrations with fiber $S^{q-1}$~\cite[Chapter 3A]{MadMil-book}.

Let $B\widetilde{PL}_q$ be a classifying space 
for block bundles with fiber $I^q$, 
it is a locally finite simplicial complex~\cite[Section 2]{RS-I}.

Taking boundary of a block bundle defines a canonical map
of $B\widetilde{PL}_q$ to $B\tilde{G}_q$, the block version of
$BG_q$ which deformation retracts onto $BG_q$~\cite[Corollaries 5.8-5.9]{RS-I}.

If $q\le 2$, the canonical maps $BO_q\to B\widetilde{PL}_q\to B\tilde{G}_q$ are 
homotopy equivalences~\cite[p.127]{Wal-book}, which is why {\em for the rest of this section we
assume $q\ge 3$}.

The set of equivalence classes of thickenings of $M$ of 
codimension $q\ge 3$ is bijective to the set of homotopy classes of maps
from $M$ into $B\widetilde{PL}_q$~\cite[Corollary 4.6]{RS-I}. 

We wish to enumerate oriented block bundles up to finite ambiguity in terms of their characteristic classes
in the same way as an oriented vector bundle is determined up to finite ambiguity by its Euler 
and Pontryagin classes. (The orientability assumption is to avoid twisted Euler class). 
To this end we consider the classifying space 
$BS\widetilde{PL}_q$ for oriented thickenings of codimension $q\ge 3$. 
Its basic properties are folklore, and not treated in the literature, so for completeness
we derive them below. 
Define $BS\widetilde{PL}_q$ as the universal cover of $B\widetilde{PL}_q$;
as we show below the covering projection is $2$-fold, and coincides with the 
forget orientation map. 

Following~\cite[Chapter 3A]{MadMil-book} we review the properties of $SG_{q}$, 
the space of degree one homotopy self-equivalences of $S^{q-1}$, where we assume $q\ge 3$:
\begin{itemize}
\item The evaluation map $SG_{q}\to S^{q-1}$ is a fibration 
whose fiber %$SF_{q-1}$ 
can be identified with a path-component of $\Omega^{q-1} S^{q-1}$.
All components of $\Omega^{q-1} S^{q-1}$ are homotopy equivalent, and
for any choice of a basepoint $\pi_i(\Omega^{q-1} S^{q-1})\cong\pi_{i+q-1}(S^{q-1})$.
\item 
$SG_q$ is the identity component in the space $G_q$ of homotopy self-equivalences of 
$S^{q-1}$; the space $G_q$ has two path-components.
\end{itemize}

The space $BSG_q$ is the classifying space for oriented fibrations
with fiber $S^{q-1}$. Since $SG_q$ is path-connected, $BSG_q$ is simply-connected.
The fiberwise join with the trivial $S^0$ bundle defines the stabilization map
$BSG_q\to BSG_{q+1}$ with direct limit $BSG$, which is a simply-connected
space with finite homotopy groups.

Similarly, $BG_q$ is the classifying space for fibrations with fiber $S^{q-1}$,
and $BG$ is the direct limit of $BSG_q$ under stabilization. 
The spaces $BG_q$, $BG$ are path-connected, and 
since $G_q$ has two path-components, the stabilization indices
an isomorphism $\pi_1(BG_q)\cong\pi_1(BG)\cong\Z_2$.

Product with $I$ defines the stabilization map
$B\widetilde{PL}_q\to B\widetilde{PL}_{q+1}$ with direct limit $B\widetilde{PL}$,
which by~\cite[Corollary 5.5]{RS-III} is homotopy equivalent
to $BPL$, the classifying space for stable PL microbundles.

To unclutter notations we sometimes suppress differences between
classifying spaces and their homotopy equivalent block versions, e.g.,
$B\tilde{G}_q$, $B\widetilde{PL}$ and $BG_q$, $BPL$, respectively.

The homotopy fiber $PL/O$ of the canonical map $G/O\to G/PL$  
is $6$-connected~\cite[Remark 4.21]{MadMil-book} and
$G/O$ is simply-connected~\cite[p.43]{MadMil-book}.
Thus $\pi_1(G/PL)=0$, or equivalently, the pair
$(B\widetilde{G}, B\widetilde{PL})$ is $2$-connected.
It follows that $BPL\to BG$ is a $\pi_1$-isomorphism.

For $q\ge 3$ it is proved in~\cite[Theorem 1.10]{RS-III} that 
the inclusion 
\[
(B\widetilde{G}_{q}, B\widetilde{PL}_q)\to (B\widetilde{G}_{q+1}, B\widetilde{PL}_{q+1})
\] 
is an isomorphism on homotopy groups, and then
the previous paragraph gives \mbox{$2$-connectedness} of the pair 
$(B\widetilde{G}_{q}, B\widetilde{PL}_q)$ for $q\ge 3$.
Hence every map in the square below induces a $\pi_1$-isomorphism
\[
\xymatrix{
B\widetilde{PL}_q\ar[r]\ar[d] & B\widetilde{G}_q \ar[d]
\\
B\widetilde{PL}\ar[r] & B\widetilde{G} 
}
\] 
where all fundamental groups have order two, 
and therefore, we get the corresponding square of the $2$-fold (universal) covers:
\begin{equation*}
\xymatrix{
BS\widetilde{PL}_q\ar[r]\ar[d] & BS\widetilde{G}_q \ar[d]
\\
BS\widetilde{PL}\ar[r] &BSG 
}
\label{diagram: BSPL_q}
\end{equation*} 
where $BS\widetilde{PL}$ is homotopy equivalent to the classifying space $BSPL$
of stable oriented PL bundles.

The canonical map $BSO\to BSPL$ is a rational homotopy equivalence,
as a map of simply-connected spaces whose homotopy fiber $PL/O$ has finite homotopy groups. 

A thickening of $M$ 
of codimension $q\ge 3$ is orientable if and only if its classifying map 
$f\co M\to B\widetilde{PL}_q$
postcomposed with $B\widetilde{PL}_q\to B\widetilde{G}_q$ lifts to $BS\widetilde{G}_q$.
Thus the thickening is orientable if and only if $f$ lifts to 
$BS\widetilde{PL}_q$.

For $q\ge 3$ it is proved in~\cite[Theorem 1.11]{RS-III} that 
the above map 
\[
(B\widetilde{PL}_{q+1}, B\widetilde{PL}_q)\to (B\widetilde{G}_{q+1}, B\widetilde{G}_q)
\] 
is an isomorphism on homotopy groups, and hence the same is true for
the corresponding map
$(BS\widetilde{PL}_{q+1}, BS\widetilde{PL}_q)\to (BS\widetilde{G}_{q+1}, BS\widetilde{G}_q)$
of $2$-fold covers.
In other words, by e.g.~\cite[Remark 3.3.19]{MunVol-book}, the square 
\begin{equation*}
\xymatrix{
BS\widetilde{PL}_q\ar[r]\ar[d] & BS\widetilde{G}_q \ar[d]
\\
BS\widetilde{PL}\ar[r] &BSG 
}
\end{equation*} 
is homotopy Cartesian. 
For Cartesian squares~\cite[Example 3.3.14]{MunVol-book}
the arrows from the upper left corner give a fibration
$BS\widetilde{PL}_q\to BSPL\times BSG_q$ whose fiber $\Omega BSG$
has finite homotopy groups~\cite[Corollary 3.8]{MadMil-book}, so that
the fibration is a rational homotopy equivalence.

The classifying spaces $BS\widetilde{PL}_q$, $BSPL$, $BSG_q$ are simply-connected with finitely generated homotopy groups 
(as can be seen from the long exact homotopy sequences of the above fibrations together with the fact that
$\pi_i(BSO)$ and $\pi_i(S^m)$ are finitely generated for all $m$, $i$). 
Hence $BS\widetilde{PL}_q$, $BSPL$, $BSG_q$ have
finitely generated homology groups
(by the Hurewicz theorem for the Serre class of finitely generated abelian groups).

In what follows we need to compare the cohomology of these classifying spaces
with coefficients in $\Z$, $\Z\!\left[\frac{1}{2}\right]$, $\Q$, for which
we recall the following.

\begin{lem} 
\label{lem: torsion coeff}
If $X$ be a space with finitely generated homology groups, then
for any short exact sequence $1\to A\to B\to G\to 1$ 
with torsion group $G$ the homomorphism $H^i(X;A)\to H^i(X;B)$ in the corresponding 
long exact cohomology sequence becomes an isomorphism after tensoring with $\Q$.
\end{lem}
\begin{proof}
For any abelian group $G$
the kernel $\mathrm{Ext}(H_{i-1}(X), G)$
of the natural surjection $H^i(X; G)\to \mathrm{Hom}(H_i(X), G)$ is torsion,
because $H_{i-1}(X)$ is finitely generated. 
Since $H_i(X)$ is finitely generated and $G$ is a torsion, the group
$\mathrm{Hom}(H_i(X), G)$ is torsion. Thus $H^i(X; G)$ is torsion. 
The exactness  of the long exact cohomology sequence corresponding
to $1\to A\to B\to G\to 1$ finishes the proof.
\end{proof}

The cohomology algebra $H^*(BSPL;\Q)$ is a polynomial algebra over $\Q$ on 
the {\scshape pl}  Pontryagin classes $p_i$, which 
%can be defined as classes that
correspond to the usual Pontryagin classes under
the rational homotopy equivalence $BSO\to BPSL$~\cite[4.20]{MadMil-book}.
The {\scshape pl} 
Pontryagin classes are generally not integral.
%, i.e., they do not lie in image of 
%the homomorphism $H^*(BSPL;\Z)\to H^*(BSPL;\Q)$ induced by the inclusion $\Z\to\Q$.
The discussion after~\cite[Theorem 11.14]{MadMil-book} gives  
explicit classes $2^iR_{4i}\in H^{4i}(BSPL)/\text{Tors}$, where $i\in \Z$, $i>0$, 
whose representatives in $\rho_i\in H^{4i}(BSPL)$ generate 
$H^*(BSPL;\L)$ as a polynomial algebra over $\L:=\Z\!\left[\frac{1}{2}\right]$. 
By Lemma~\ref{lem: torsion coeff} the coefficient homomorphisms 
\begin{equation}
\label{form: composite}
H^*(BSPL;\Z)\to H^*(BSPL;\L)\to H^*(BSPL;\Q)
\end{equation}
become isomorphisms after tensoring with $\Q$.
Then the images of $\rho_i$ under the composite (\ref{form: composite}) 
generate $H^*(BSPL;\Q)$ as a polynomial algebra over $\Q$.

We keep the notations $\rho_i$ and $p_i$ for their images in
$H^*(BS\widetilde{PL}_q)$ under the stabilization $BS\widetilde{PL}_q\to BSPL$.

Milnor in~\cite{Mil} showed that $H^*(BSG_q;\Q)$ is a 
polynomial algebra over $\Q$ on a single generator $o_q$
whose degree $|o_q|$ equals $q$ is even, and $2q-2$ if $q$ is odd.
By Lemma~\ref{lem: torsion coeff} the coefficient homomorphism 
$H^*(BSG_q;\Z)\to H^*(BSG_q;\Q)$ becomes an isomorphism after tensoring with $\Q$.
Hence we can (and will) pick $o_{q}$ to be an integral class,
and if $q$ is even we normalize $o_q\in H^q(BSG_q;\Q)$ to be the 
integral Euler class of the universal spherical fibration under
the coefficient homomorphism $H^q(BSG_q)\to H^q(BSG_q;\Q)$.
For any $k\in\Z_{>0}$ we have:
\begin{enumerate}
\item
The canonical map $BSO_{2k+1}\to BSG_{2k+1}$ takes $o_{2k+1}$
to a nonzero multiple of the Pontryagin class $p_k$~\cite[pp.73--74]{Mil}. 
\item
The stabilization map $BSG_{2k}\to BSG_{2k+1}$ 
takes $o_{2k+1}$ to a nonzero rational multiple of $o_{2k}^2$~\cite[pp.73--74]{Mil}.
\item 
The stabilization map $BSG_{2k+1}\to BSG_{2k+2}$ takes $o_{2k+2}$ to zero.
(This is obvious by degree reasons except when $k=1$ which is excluded as follows: 
if the stabilization $\s\co BSG_{3}\to BSG_{4}$ is nonzero on $H^4(-;\Q)$, then 
$\s$ is a rational homotopy equivalence of rationally $3$-connected spaces, and hence 
$\s^*o_4$ is nonzero on some $f\co S^4\to BSG_{3}$,
so that the Euler class $o_4$ is nonzero on $\s\circ f$, 
which is a sphere bundle with a section).
Thus the double stabilization $BSG_{2k+1}\to BSG_{2k+3}$ 
takes $o_{2k+3}$ to zero.
\end{enumerate}

For $\xi\co M\to BS\widetilde{PL}_q$ let $S(\xi)$ be the spherical 
fibration obtained by composing $\xi$ with $BS\widetilde{PL}_q\to BSG_q$
and define $o_q(\xi):=o_q(S(\xi))$.  
It follows that (2) and (3) hold with $G$ replaced by $\widetilde{PL}$. 

The Euler class $e$ of an oriented thickening defined in
Section~\ref{sec: thick} is sent by
the coefficient homomorphism $H^q(M)\to H^q(M;\Q)$
to $o_q$ if $q$ is even, and to $0$ if $q$ is odd
because in this case $H^q(BSG_q;\Q)=0$.

The sequence $(o_q,\rho_1, \rho_2, \dots )$ defines the map 
\begin{equation}
\label{eq: pl tilde rhe}
BS\widetilde{PL}_q\longrightarrow \bm{K}_q:=K(\Z, |o_q|)\times\prod_{i\in\Z_{>0}} K(\Z, 4i).
\end{equation}
where the co-domain is topologized as a weak product~\cite[p.28]{Whi78}. The map
is a rational homotopy equivalence (because the rational
cohomology of the domain and the co-domain are 
polynomial algebras over $\Q$ and generators are sent to generators).

\begin{thm}
\label{thm: realization}
Given a closed manifold $M$, an integer $q\ge 3$, and a sequence of cohomology classes 
$(\b_i)_{i\in\Z_{\ge 0}}$ with $\b_0\in H^{|o_q|}(M)$ and $\b_i\in H^{4i}(M)$ 
for $i\in\Z_{>0}$, there exists an $n\in\Z_{>0}$ and an oriented codimension $q$ 
thickening $W$ of $M$ with $o_q(W)=n\b_0$ and $\rho_i(W)=n\b_i$.
\end{thm}
\begin{proof}
Realize the sequence $(\b_i)$ as a map from $M$ into $\bm{K}_q$
and apply Theorem~\ref{thm: prescribe LR}
to lift the map into $BS\widetilde{PL}_q$.
\end{proof}

\begin{rmk}
\label{rmk: realize p_i}
The classes $\rho_i$ and the Pontryagin classes $p_j$ are related by universal polynomials,
so to some extent Theorem~\ref{thm: realization} lets us specify the 
Pontryagin classes, albeit in a non-explicit way. 
Still, if we insist that $p_i=0$ for $i\neq j$, then $\rho_i$ and
$p_i$ are proportional, and then $o_q$, $p_i$ can be prescribed
arbitrarily up to multiplicative constant. 
\end{rmk}

Since the fiber of (\ref{eq: pl tilde rhe}) is rationally contractible, 
standard obstruction theory arguments imply:

\begin{thm}
For a closed manifold $M$ and an integer $q\ge 3$, there are at most
finitely many equivalence classes of codimension $q$ thickenings of $M$ with
the same classes $o_q$, $p_i$, $i\in\Z_{>0}$.
\end{thm}
 
Let us highlight some
differences between $BSO_q$ and $BS\widetilde{PL}_q$:
\begin{itemize}
\item 
The Pontryagin class $p_i$ in $H^*(BSO_q; \Q)$ vanishes $i>\frac{q}{2}$
while the {\scshape pl} Pontryagin class $p_i$ in 
$H^*(BS\widetilde{PL}_q; \Q)$ is nonzero for any $i$.
\item 
The relation $o_{2k}^2=p_k$ holds in $H^*(BSO_{2k})$ while
in $H^{4k}(BS\widetilde{PL}_q; \Q)$ the classes $o_{2k}^2$,
$p_k$ are linearly independent.
\item
The canonical map $BSO_{2k+1}\to BS\widetilde{PL}_{2k+1}$ takes 
$o_{2k+1}$ to a nonzero multiple of the Pontryagin class $p_k$.
Thus for a linear $(2k+1)$-disk bundle $p_k=0$ if and only if $o_{2k+1}=0$.
\end{itemize}

\section{Smoothable thickenings of codimension $\ge 3$}
\label{sec: Smoothable thickenings}

Theorem~\ref{thm: realization} allows us to construct thickenings whose characteristic classes
are prescribed up to a multiplicative constant. More work is required to produce thickenings
that are smoothable. By smoothing theory~\cite{HM74} a manifold $W$ is smoothable
if and only if its tangent microbundle $\tau_W\co W\to BPL$ lifts to $BO$.

\begin{thm} 
\label{thm: smoothable thick of codim >2}
If in Theorem~\ref{thm: realization} the manifold $M$ is smoothable,
then the integer $n$ can be chosen so that the thickening $W$ is smoothable.
\end{thm}
\begin{proof}
Apply Theorem~\ref{thm: prescribe LR} for $X=M$ and
\vspace{5pt}\newline
$\phantom{\quad}$  $\a_1\co BS\widetilde{PL}_q\to BSG_q\times BSPL$ is 
the above rational homotopy equivalence,\vspace{3pt}\newline
$\phantom{\quad}$ $\a_2$ is the product of
the identity map of $BSG_q$ with $BSO\to BSPL$,\vspace{3pt}\newline 
$\phantom{\quad}$ $f\co M\to\bm{K}_q$ 
represents the homotopy class given by the sequence
$(\b_0,\b_1,\dots )$,\vspace{3pt}\newline
$\phantom{\quad}$ $\beta\co BSG_q\times BSPL\to \bm{K}_q$ 
represents the homotopy class given by 
$(o_q,\rho_1, \rho_2, \dots )$, \vspace{-7pt}\newline 
to get $f_1, f_2$ such that $\a_1\circ f_1$, $\a_2\circ f_2$
are homotopic. The homotopy class $[f_1]$ classifies a thickening $W$ 
whose stabilization is the composite of $\a_1\circ f_1$ with 
the coordinate projection to $BSPL$. By commutativity
the stabilization of $[f_1]$ lifts to $BSO$.
Thus the tangent microbundle $\tau_W$ of $W$ 
is such that $\tau_W\oplus\nu_M$ has a structure of a stable vector bundle,
and hence so does 
$\tau_W\oplus\nu_M\oplus\tau_M$ which is stably isomorphic to $\tau_W$.
Here $\nu_M$ is the stable normal bundle of $M$.
\end{proof}

\begin{proof}[Proof of Theorem~\ref{thm: non-smooth} for $q\ge 3$]
The construction depends on the parity of $q$.

{\bf $q$ is even}. Fix $k, l\in\Z$ with $l\ge 4k\ge 8$, set $q=2k$, 
and let $L$ be any smooth closed $l$-manifold such that 
$H^{2q}(L)$ contains an infinite order element.
By Theorem~\ref{thm: smoothable thick of codim >2} and 
Remark~\ref{rmk: realize p_i} there is a
 smoothable thickening $W$ of $L$ with $p_k\neq 0$ and 
$o_q=0=p_i$ for $i\neq k$.

{\bf $q$ is even}
Fix $k, l\in\Z$ with $l\ge 4k\ge 4$, set $q=2k+1$, and let $L$ 
be any smooth closed \mbox{$l$-manifold} such that $H^{4k}(L)$ 
contains an infinite order element. 
By Theorem~\ref{thm: smoothable thick of codim >2} and 
Remark~\ref{rmk: realize p_i} there is
a smoothable thickening $W$ of $L$ such that $o_q\neq 0=p_i$ for all $i$. 

Regardless of the parity of $q$ let $K$ be the double of $W$ along the boundary.
Apply Ontaneda's Riemannian hyperbolization to a smooth triangulation
of $K$ to get $\bm L\subset\bm K$ satisfying (i)-(ii).
Let $V$ be a covering space
of $\bm K$ corresponding to the image of $\pi_1(\bm L)$ in $\pi_1(\bm K)$, and
let $f$ denote a lift of the inclusion $\bm L\to\bm K$ to $V$. 
By Theorem~\ref{thm: pullback} and Section~\ref{sec: thick}
the classes $o_q$, $p_i$ of the regular neighborhood of $f(\bm L)$ in $V$ 
satisfy the same identities (depending on parity of $q$) as those of $W$,   
and hence so does the normal disk bundle of the smooth submanifold
that is a deformation retract of $V$. The existence
of such a smooth submanifold leads to
a contradiction as follows:

{\bf $q$ is even}.
Linear disk bundles satisfy $o_q^2=p_k$ which 
gives $0=o_q^2=p_k\neq 0$ in our case.

{\bf $q$ is odd}.
The map $BSO_q\to BS\widetilde{PL}_q$
sends the {\scshape pl} Pontryagin class to 
the usual Pontryagin class $p_k$, and takes $o_q$ to
a multiple of $p_k$. Thus $o_q$ of a linear disk bundle
with $p_k=0$ must be zero, which contradicts $o_q\neq 0$.

Proposition~\ref{prop: finite cover}
proves the claim about the finitely-sheeted covers of $\overbar{\bm{K}}$
because finitely-sheeted covers are injective on rational cohomology.
\end{proof}

\section{Classifying spaces of codimension $2$ thickenings}
\label{sec: classifying spaces of codimension 2 thickenings}

A classifying space $BRN_2$ for codimension $2$ thickenings was introduced 
in~\cite{CS76} by considering them up to concordance.
Concordant thickenings are homeomorphic rel $M$, and the set of concordance classes 
of codimension $2$ thickenings 
is bijective to the set of homotopy classes of maps
from $M$ into a classifying space $BRN_2$, see~\cite{CS76} and 
also~\cite{Mc-cone}.

A {\em concordance of oriented thickenings\,} is defined
in~\cite[p.173]{CS76}, where it is shown that 
the set of concordance classes of oriented codimension $2$ thickenings 
is bijective to the set of homotopy classes of maps
from $M$ into a classifying space $BSRN_2$, which was extensively discussed 
in~\cite{CS76, CapSha-symp}.
The space $BSRN_2$ is a path-connected CW complex~\cite[Theorem 1.4]{CS76}.

To get smoothable codimension $2$ thickenings with prescribed Euler and 
Pontryagin classes consider the diagram
\begin{equation*}
\label{diagram: BSRN_2}
\xymatrix{
BSPL & BSO\times BSPL\ar[l]^-{\oplus}\ar[l] &  
\s(\ast)\times BSO \ar[l]
\\
BSRN_2\ar@<.2pc>[r]^-{(e,\eta)}\ar[u]^-{\s} & 
BSO_2\times G/PL\ar[u]_{\iota=\s\times\pi}
\ar@<.2pc>[l]^-{s}
& \ast\times G/O\ar[u]\ar[l]
}
\end{equation*} 

Here $\ast$ is a basepoint in $BSO_2$, $\oplus$ is the Whitney sum,  
the map $\iota$ is the product of the 
stabilization $\s\co BSO_2\to BSO$ and the standard fibration
$\pi\co G/PL\to BSPL$ with fiber $SG$.
The rightmost upward arrow is the standard fibration 
with fiber $SG$. The rightmost top arrow is the product of the basepoint inclusion 
with the forget-$SO$-structure map.
The rightmost bottom arrow is the product of the basepoint inclusion
with the standard fibration $G/O\to G/PL$ which is a rational homotopy
equivalence (whose fiber $PL/O$ is $6$-connected and $\pi_i(PL/O)$ is
a finite group of h-cobordism classes of oriented homotopy $i$-spheres).
The map $e$ corresponds to
a boundary-preserving homology equivalence that sends
a thickening to a \mbox{$2$-disk} bundle with the same Euler class, 
and $\eta$ corresponds to its normal invariant.
 
The rightmost square commutes by construction. 
Corollary 4.9 on~\cite[p.208]{CS76} says that $(e,\eta)$
has a section $s$ that extends the map $BSO_2\times\ast\to BSRN_2$ 
that considers a $2$-disk bundle as a thickening. 
Proposition 1.8 on~\cite[p.183]{CS76} gives 
commutativity of the leftmost square, i.e., $\s=\oplus\circ\iota\circ (e,\eta)$,
which also implies $\s\circ s=\oplus\circ\iota\circ (e,\eta)\circ s=\oplus\circ\iota$.

Every pair of continuous maps $f\co M\to BSO_2$ and $g\co M\to G/PL$ defines a
thickening $W$ classified by $s\circ (f\times g)\co M\to BSRN_2$ whose stabilization  
\begin{equation}
\label{form: class map}
(\s\circ f)\oplus(\pi\circ g)\co M\to BSPL
\end{equation} 
lifts to $BSO$ if $g$ lifts to $G/O$.
Therefore, if $g$ lifts to $G/O$ and $M$ is smoothable, then $W$ is smoothable.
Now we are ready to prove the following.

\begin{thm} 
\label{thm: realization codim 2}
Given a closed smooth manifold $M$ and a sequence $(\b_i)_{\Z_{\ge 0}}$ 
with $\b_0\in H^2(M)$ and $\b_i\in H^{4i}(M)$ for $i\in\Z_{>0}$  
there is a positive integer $n$ and 
a smoothable oriented codimension $2$ thickening $W$ over $M$ with 
$e(W)=\b_0$ and $p_i(W)=n\b_i$.
\end{thm}
\begin{proof}
The standard fibration $G/O\to BSO$ with fiber $SG$
is a rational homotopy equivalence. Integral Pontryagin classes 
define a rational homotopy equivalence $BSO\to K$ where $K$ is the product
of Eilenberg-MacLane space in positive degrees divisible by $4$.
Realize the sequence $(\b_i)$ as a map $M\to K$, and 
apply Theorem~\ref{thm: prescribe LR} to the composite $G/O\to BSO\to K$, 
which gives a lift $g\co M\to G/O$ whose composite with $G/O\to BSO$ has
Pontryagin classes $p_i=n\b_i$. Any element in $H^2(M)$ is an Euler class
of a vector bundle $f\co M\to BSO_2$. Applying $s$ as in the above diagram
gives a smoothable thickening
with desired Euler and Pontryagin classes.
\end{proof}

\begin{proof}[Proof of Theorem~\ref{thm: non-smooth} for $q=2$ and $l\ge 4$]
Let $L$ be any smooth closed oriented \mbox{$l$-manifold} such that $H^{4}(L)$ 
contains an infinite order element. Use Theorem~\ref{thm: realization codim 2}
to find a codimension $2$ smoothable thickening with $e=0\neq p_1$.
If $\overbar{\bm K}$ deformation retracts to a smooth
$l$-dimensional submanifold, its normal bundle also has
$e=0\neq p_1$, which contradicts $e^2=p_1$. Proposition~\ref{prop: finite cover}
proves the claim about the finitely-sheeted covers of $\overbar{\bm{K}}$
because finitely-sheeted covers are injective on rational cohomology.
\end{proof}

\begin{proof}[Proof of Theorem~\ref{thm: non-smooth} for $q=2$ and 
$l\in\{2,3\}$]
Recall that $\pi_2(G/PL)\cong\pi_2(G/O)\cong\Z_2$. The map 
$(\eta, e)\co BSRN_2\to BSO_2\times G/PL$ has a cross-section,
so there is a codimension $2$ oriented thickening $R$ 
of $S^2$ with non-trivial normal invariant and Euler class $e$.
The composite $S^2\to BSRN_2\to G/PL$ is surjective on $\pi_2$, 
and hence on $H_2$ by the Hurewicz theorem.

If $m=2$, let $L=S^2$ and $K$ be the double of $R$ along the boundary.

If $m=3$, let $L=S^2\times S^1$ and $K$ be the double of $R\times S^1$ 
along the boundary. 

Note that $R\times S^1$ is the regular neighborhood of $L$ that
is the pullback of $R$ via the coordinate projection 
$S^2\times S^1\to S^2$, which is surjective on $H_2$.

Apply Ontaneda's Riemannian hyperbolization to smooth triangulation of $K$ 
for which $L$ is a subcomplex. Lift the inclusion $\bm L\to \bm K$ to 
a covering of $\bm K$ corresponding to $\pi_1(\bm L)$.
Since $\pi_1(\bm L)$ is quasiconvex, the covering space 
is the interior of the compact manifold $W$ with boundary that is 
the $\e$-neighborhood of the convex core, and we still denote by $\bm L$ its lift in $W$.

The regular neighborhood of $\bm L$ in $W$ is equivalent to
the regular neighborhood of $\bm L$ in $\bm K$, which by
Theorem~\ref{thm: pullback} is the pullback via $h$
of the regular neighborhood of $L$ in $K$. 
Since $h$ is surjective on $H_2$, so is the composite
\[
\bm L\to L\to S^2\to BSRN_2\to G/PL,
\]
where $L\to S^2$ is the coordinate projection if $m=3$.
Therefore, the regular neighborhood of $\bm L$ in $W$ has a 
nontrivial normal invariant.
Similarly, if $m=2$, then 
the normal invariant of $\bm L\times S^1$ in $W\times S^1$
is nontrivial (again because the coordinate projection 
$\bm L\times S^1\to \bm L$ is $H_2$-surjective).

Set $(W^\prime, L^\prime)=(W,\bm L)$ if $m=3$ and 
$(W^\prime, L^\prime)=(W\times S^1,\bm L\times S^1)$ if $m=2$.
Thus $W^\prime$ is a compact $5$-manifold
that deformation retracts onto a {\scshape pl} embedded 
closed $3$-manifold $L^\prime$
whose regular neighborhood has nontrivial normal invariant.

Arguing by contradiction suppose that $W^\prime$ deformation retracts to a 
smooth submanifold $M^\prime$. The restriction to $L^\prime$
of the deformation retraction $W^\prime\to M^\prime$ is homotopic to a diffeomorphism 
$g\co L^\prime\to M^\prime$ (which uses the geometrization if $m=3$). 
Hence the normal invariant of $g^{-1}$ is trivial. 

As we explain in Appendix~\ref{app: poincare emb},
the pullback via $g^{-1}$
of the Poincar{\'e} embedding given by the inclusion 
$L^\prime\subset W^\prime$
is isomorphic to the Poincar{\'e} embedding of the inclusion 
$M^\prime\subset W^\prime$,
By~\cite[Theorem 6.2]{CS76} the Poincar{\'e} embedding for  
$M^\prime\subset W^\prime$ can be realized by a 
locally flat embedding if and only if the normal invariants
of $g^{-1}$ equals the normal invariant of
the Poincar{\'e} embedding $L^\prime\subset W^\prime$, 
which is not the case. 
\end{proof} 

\begin{proof}[Proof of Theorem~\ref{thm: non-smoothable spine}]
Let $L$ be any closed connected {\scshape pl} $l$-manifold $L$ with some non-integer Pontryagin number. For any hyperbolized simplex the strict hyperbolization $\bm{L}$ has the same property.
Any homotopy equivalence from $\bm L$ to a closed manifold
preserves Pontryagin classes, see Remark~\ref{rmk: borel},
and hence Pontryagin numbers, which are integers for 
smooth manifolds. Thus $\bm L$ is not homotopy 
equivalent to a smooth manifold.

Fix a {\scshape pl} embedding of $L$ into a high-dimensional sphere $K$. 
Let $(\bm{K}, \bm{L})$ be the result of applying 
Ontaneda's strict hyperbolization to a {\scshape pl} 
triangulation of the pair 
$(K, L)$. 

Closed {\scshape pl} manifolds with non-integer Pontryagin number exist in all 
dimensions $l\ge 8$ divisible by $4$, see~\cite[section 5]{BLW}. There are
such simply-connected manifolds, see~\cite[V.2.9]{Bro-book}, and any closed 
simply-connected {\scshape pl} $l$-manifold with $l\ge 4$
admits a {\scshape pl} embedding into the $(2l-1)$-sphere~\cite{Irw}.
\end{proof}

\section{Hyperbolization pulls back regular neighborhood}
\label{sec: hyperb pulls back} 

Recall that unless stated otherwise 
all maps, manifolds, embeddings, isotopies, bundles, submanifolds are {\scshape pl}. 

To simplify notations for a map $f\co X\to\, \triangle$
and $J\subseteq \triangle$ we set $X_J:=f^{-1}(J)$. 

An $n$-dimensional {\em hyperbolized simplex\,}
is a map $f\co (X, \d X)\to\, (\triangle, \d\triangle )$ such that
$\triangle$ is an \mbox{$n$-simplex}, 
$X$ is a compact connected $n$-manifold with boundary,  
$f$ has degree one mod $2$, and 
for every $k$-dimensional face $\a$ of $\triangle$ the set
$X_\a$ is a \mbox{$k$-dimensional} submanifold of $X$
whose boundary equals $X_{\partial\a}$.
This definition corresponds to conditions C0, C1, C2 
in~\cite[Section 1c]{DavJan}. 
We always assume that the hyperbolized simplex is {\em proper\,} 
as defined in Appendix~\ref{app: transversality}.
Then the restriction $f$ to a component of $X_\a$ is a proper
hyperbolized simplex.

\begin{rmk} In our applications $X$ is oriented, 
the map $f\co (X, \partial X )\to (\triangle, \partial\triangle)$ has
degree one, and the 
tangent bundle to $X$ has zero rational Pontryagin classes.
\end{rmk}

For $0\le l\le k$ let $\vartriangle^l$ denote 
the convex combination of $e_1, \dots, e_{l+1}$ in $\mathbb R^{k+1}$;
thus $\vartriangle^l$  is an $l$-dimensional face of the standard $k$-simplex
$\vartriangle^k$. 

Let $K$ be a $k$-dimensional simplicial complex. Denote its  
first barycentric subdivision by $K^\prime$, and
let $p_{\!_K}\co K^\prime\to\vartriangle^k$ be the simplicial map 
defined on vertices by sending the barycenter of every $i$-dimensional simplex of $K$ 
to $e_{i+1}$. The map $p_{\!_K}$ is injective on every simplex, and thus
it ``folds'' $K^\prime$ onto $\vartriangle^k$.
Then
\[
\bm{K}_f:=\{(s,x)\in K\times X\co p_{\!_K}(s)=f(x)\}
\]
is the fiber product of $p_{\!_K}$ and $f$.
We refer to $h_{\!_K}\co\bm{K}_f\to K$ given by $h_{\!_K}(s,x)=s$ as the {\em hyperbolization map}.
If $K$ is a closed manifold, then so is $\bm{K}_f$.

Let $L$ be a full $l$-dimensional subcomplex of $K$. Then $p_{\!_L}$ is the
restriction of $p_{\!_K}$ to $L$. Let $f_l$ be the restriction of $f$ to a component 
$X^l$ of $f^{-1}(\vartriangle^l)$; recall that $f_l$ is a hyperbolized $l$-simplex.
The fiber product $\bm{L}_{f_l}$ of $p_{\!_L}$ and 
$f_l$ can be also described as  
\[
\bm{L}_{f_l}= 
\{(s,x)\in L\times X\co p_{\!_L}(s)=f(x)\}
\]
because the equation $p_{\!_L}(s)=f(x)$ has no solutions outside $L\times X^l$.

\begin{proof}[Proof of Theorem~\ref{thm: pullback}]
We are going to quote results that require 
$\bm{K}$ to stay away from the boundary of the ambient manifold.
To this end we attach a collar to $K\times X$. Let $\bar{X}$ be the result
of attaching $\d X\times [0,1]$ along the boundary of $X$.
Let $b$ be the barycenter of $\vartriangle^k$, and
let $\bar\vartriangle^k$ be the image of $\vartriangle^k$ under the self-map
of $\R^{k+1}$ given by $x\to b+2(x-b)$; thus $\bar\vartriangle^k\supset\vartriangle^k$ 
are concentric simplices in the affine $k$-plane spanned by $\vartriangle^k$.
We extend $f$ to the map $\bar f\co\bar{X}\to \bar\vartriangle^k$
as follows: pass to subdivisions in which $f$ is simplicial, take an arbitrary
simplex $\s$ in the subdivision of $\d X$, and send the prism $\s\times [0,1]$
linearly to the portion of the cone over $f(\s)$ with tip $b$
that lies in $\bar\vartriangle^k\setminus\Int(\vartriangle^k)$.

Thus $\bm{K}$ is a locally flat submanifold of $\mathcal K$
with trivial normal bundle $\nu_{\bm K}^{\,\mathcal K}$,
cf. Appendix~\ref{app: transversality}. 
Note that $\bm{K}$ lies in $K\times X$
and intersects its boundary, and we attached collars precisely
to keep $\bm K$ away from the boundary.
Recall that $\bm{L}$ 
is the intersection of $\bm{K}$ and $\mathcal L:=L\times \bar X$.
By Appendix~\ref{app: transversality}
we can (and will) choose $\nu_{\bm K}^{\,\mathcal K}$ 
so that its fibers over $\bm L$ contains the fibers of 
the normal bundle $\nu_{\bm L}^{\,\mathcal L}$ of 
$\bm{L}$ in $\mathcal L$.

Suppose $k-l>2$. 
The normal block bundle $\nu_{\mathcal L}^{\,\mathcal K}$ of $\mathcal L$
in $\mathcal K$ is the pullback of the normal block bundle $\nu_{L}^{K}$ 
of $L$ in $K$ via the coordinate projection $\pi_L\co\mathcal L\to L$.
The hyperbolization map $h_L$ factors as the inclusion $\bm{L}\to\mathcal L$
followed by $\pi_L$, and therefore, $h_L$ pulls  
$\nu_{L}^{K}$ back to $\nu_{\mathcal L}^{\,\mathcal K}\vert_{\bm{L}}$.
By~\cite[Theorem 1.2]{RS-II} there is a small isotopy of $\mathcal K$
that takes $\bm{K}$ to a submanifold $\bm{K}^t$ that 
is block transverse to $\mathcal L$, 
i.e., $\bm{L}^t:=\bm{K}^t\cap \mathcal L$ is a submanifold, and
the intersection of $\bm{K}^t$ and $\nu_{\mathcal L}^{\,\mathcal K}$ 
equals $\nu_{\mathcal L}^{\,\mathcal K}\vert_{\bm{L}^t}$, 
the normal block bundle of $\bm{L}^t$ in $\bm{K}^t$.
(Here we do not distinguish between a block bundle and its total space).
Since the above isotopy is small,
the projection of $\nu_{\bm K}^{\,\mathcal K}$ restricts to a homeomorphism  
$(\bm{K}^t, \bm{L}^t)\to (\bm{K}, \bm{L})$ identifying
the normal block bundle of $\bm{L}$ in $\bm{K}$ with 
$\nu_{\mathcal L}^{\,\mathcal K}\vert_{\bm{L}^t}$
which is isomorphic to $\nu_{\mathcal L}^{\,\mathcal K}\vert_{\bm{L}}$, as claimed.

If $k-l\le 2$, then the same argument works except that one has to replace block bundles structures
on $\nu_{\mathcal L}^{\,\mathcal K}$, $\nu_{L}^{K}$ by stratifications in the sense of Stone~\cite{Sto72}
of the regular neighborhoods of $\mathcal L$ in $\mathcal K$, and $L$ in $K$, respectively, and appeal to a more general
transversality theorem of Stone, see~\cite[Corollary on p.97]{Sto72}; 
also cf.~\cite[pp.173--175]{CS76},~\cite[p.285]{Mc75} and~\cite[p.162]{Mc-cone}.
\end{proof}

\begin{proof}[Proof of Corollary~\ref{cor: Ontaneda smooth}]
By~\cite{Ont-ihes} the manifold $\bm K$ has a smooth structure 
that admits a Riemannian metric $g_{\e, K}$ as in (b), and induces the given 
{\scshape pl} structure.
By Theorem~\ref{thm: pullback} the {\scshape pl} submanifold $\bm L$ has a 
linear disk bundle neighborhood in $\bm K$.
Therefore, by~\cite[Theorem 7.3]{LR65} there is a 
piecewise-differentiable homeomorphism $d$ of $M$ that moves $\bm L$
to a smooth submanifold $h(\bm L)$ of $\bm K$. 
Hence $\bm L$ is a smooth submanifold in the pullback, via $d$, of the 
given smooth structure on $M$. Then $d$ is a diffeomorphism of these
smooth structures, and the pullback metric $d^*g_{\e, K}$ satisfies (b).
\end{proof}

\begin{rmk}
\label{rmk: convex-cocompact}
Let us justify the claim made in the introduction that the cover of 
the Riemannian hyperbolization $\bm K$ corresponding to $\pi_1(\bm L)$
is convex-cocompact.
Since $\bm L$ is locally convex in the CAT$(-1)$ metric on $\bm K$, 
the subgroup $\pi_1(\bm L)$ of $\pi_1(\bm K)$ is quasiconvex.
Hence $\pi_1(\bm L)$ is a hyperbolic group with ideal boundary 
equivariantly homeomorphic to its limit set in the ideal boundary
of $\pi_1(\bm K)$. Every hyperbolic group acts cocompactly on the 
triple space of its boundary.
A group of isometries of a negatively  pinched Hadamard manifold
is convex-cocompact if and only if 
the group acts cocompactly on the 
space of triples of the limit set~\cite{Bow95, Bow99}.
\end{rmk}

\appendix

\section{Transversality and hyperbolization}
\label{app: transversality}

The purpose of this appendix is
to prove that {\em $\bm K$, $\bm L$ are submanifolds 
of $\mathcal K$, $\mathcal L$ that have
trivial normal bundles $\nu_{\bm K}^{\mathcal K}$, 
$\nu_{\bm L}^{\mathcal L}$ such that each fiber
of $\nu_{\bm L}^{\mathcal L}$ lies in a fiber of $\nu_{\bm K}^{\mathcal K}$.}

We follow notations and conventions of Section~\ref{sec: hyperb pulls back}, and
in particular, work in the {\scshape pl} category. 
We also assume that $K$, $L$ are closed manifolds, and
identify ${\Delta}^k$ with a \mbox{$k$-simplex} $\triangle$ in $\R^k$.
Then $p_{\!_K}-f$ becomes a map from $\mathcal K$
to $\R^k$ whose zero set is $\bm K$.
It is stated in~\cite[1f.3]{DavJan} that the map is ``transverse to $0$'', 
and hence ``$\bm K$ is a submanifold of $\mathcal K$ with trivial normal bundle''. 
No justification is given, and in fact, as the author
learned from Pedro Ontaneda, the claim is not true without
the extra assumption that $f$ is ``transverse to every face of $\triangle^k$'', 
which probably is implicit in~\cite{DavJan, ChaDav}.

Two submanifolds are {\em transverse\,} if in a local chart they become 
hyperplanes in general position; the same definition works
at a boundary point~\cite[p.436]{AZ67}. Clearly,
if $M$, $N$ are transverse, then $M\cap N$ is 
a locally flat submanifold of $M$.

Let  $\G_{\!_K}\subset\mathcal K\times\R^{2k}$ be the graph of $p_{\!_K}\times f$,
let $P_k$ be the diagonal in $\R^{k}\times\R^k=\R^{2k}$, and
let $\mathcal P_{\!_K}=\mathcal K\times P_k$.
The coordinate projection $\G_{\!_K}\to\mathcal K$ is a homeomorphism that
takes $\G_{\!_K}\cap \mathcal P$ to $\bm K$.

By Lemma~\ref{lem: loc flat} below, 
$\G_{\!_K}$ and $\mathcal P_k$ are transverse, and in particular, 
$\G_{\!_K}\cap\mathcal P_k$ is
a locally flat submanifold of $\G_{\!_K}$ of codimension $k$.
The normal bundle of $\G_{\!_K}\cap\mathcal P_k$ in $\G_{\!_K}$ 
is trivialized by the
composite of the coordinate projection $\mathcal K\times\R^{2k}\to \R^{2k}$ 
followed by the orthogonal projection  $\R^{2k}\to P_k^\bot$
onto the orthogonal complement of $P_k$ in $\R^{2k}$.

The restriction of the
orthogonal projection  $\R^{2k}\to P_k^\bot$ to $\R^{2l}$
is the orthogonal projection  $\R^{2l}\to P_l^\bot$.
Thus one can choose the trivialization of
the normal bundle of $\G_{\!_L}\cap\mathcal P_l$ 
in $\G_{\!_L}$ so that each of its fibers lies in the fiber of
the normal bundle of
$\G_{\!_K}\cap\mathcal P_k$ in $\G_{\!_K}$.

To prove Lemma~\ref{lem: loc flat} we need some terminology.
If $z$ is a point in the subspace $S$ of $Z$, and $U$ is a neighborhood of $z$ in $Z$
such that there is a map $U\to U\cap S$ that is a disk bundle over a disk
then $U$ is a {\em square at $z\in S\subset Z$}. A map of squares is 
a {\em product\,} if after composing with trivializations
of the disk bundles the map becomes the product of two maps of disks
that correspond under the trivialization to bases and fibers.

\begin{ex} 
\label{ex: square}
Let $Z$ be a triangulated compact manifold (possibly with boundary).
Give $Z$ the path-metric $d$ induced by making
every simplex isometric to the standard simplex.
Given $z\in Z$ let $\s$ be the simplex of $Z$
whose relative interior $\mathring{\s}$ contains $z$.
Let $B$ be a metric ball about $z$ in $\mathring{\s}$. 
If $\de>0$ is sufficiently small, then
the nearest point projection $\pi$ onto $\s$ 
is defined on the $\de$-neighborhood of $B$, and 
$U=\{z\in Z\,|\, d(z,\s)\le\de,\ \pi(z)\in B\}$ 
is a square lying in the star of $\s$. Moreover, 
$\pi\co U\to B$ is a disk bundle 
whose fiber is the cone on the link of $\s$ in $Z$.
Each fiber (or, rather, its intersection with any simplex
that meets $\mathring\s$) is orthogonal to $\s$.
Thus the base and the fibers are orthogonal.
\end{ex}

\begin{ex}
The folding map $p_{\!_K}$ takes every simplex $\s$ of $K^\prime$
homeomorphically onto $p_{\!_K}(\s)$. By Example~\ref{ex: square}
at every point $z\in\mathring{\s}$ there is a square sent by
$p_{\!_K}$ to a square at $p_{\!_K}(z)\in p_{\!_K}(\mathring{\s})$, 
and moreover, $p_{\!_K}$ is a product. The factor
corresponding to the map of bases is a bijection, being
the restriction of $p_{\!_K}\vert_\s$.
Since $K$ is a boundaryless manifold, 
the factor corresponding to the map of fibers is bijective only
if $z\in \mathring\triangle$, in which case the fiber is a point.
\end{ex}

\begin{ex}
Let $f\co X\to\triangle$ be a hyperbolized simplex. 
Denote the relative interior of a face $\a$ of 
$\triangle$ by $\mathring\a$. We say that $f$ is {\em proper\,} 
if for every face $\a$ of $\triangle$ with $\a\neq\triangle$ and every 
$z\in X_{\mathring\a}$ there is a square at $z\in X_{\mathring\a}$ in 
$X_{\mathring\a\,\cup\mathring\triangle}$ whose disk bundle
is the pullback via $f$ of the disk bundle of
a square at $f(z)\in \mathring\a$ in $\mathring\a\cup\mathring\triangle$
as in Example~\ref{ex: square}.
Thus $f$ is a product such that the factor corresponding to the map
of the fibers is a bijection. The factor corresponding to
the bases need not be a bijection because $f\co X_{\mathring\a}\to\mathring\a$
is not locally bijective if $\dim(X)>1$.
\end{ex} 

\begin{lem}
\label{lem: loc flat}
If $f$ is proper, then $\G$ and $\mathcal P$ are transverse. 
\end{lem}
\begin{proof}
Fix $(s_{\circ}, x_\circ)\in\bm K$, and use the disk bundle structures
of the squares at $s_\circ\in K$, $x_\circ\in X$ to 
represent points $s=(b, v)$, $x=(\b, \nu)$ of $\mathcal K$ 
near $s_\circ$, $x_\circ$, respectively. Here $b, \b$ are in the bases, while
$v, \nu$ are in the fibers of the disk bundles.

By working in a chart we will assume that both squares are in $\R^k$
with fibers and the bases parallel to coordinate planes, and moreover, 
$s_\circ$, $x_\circ$ become the origin. Denote the product of the
squares in $\R^k\times\R^k$ by $\Pi$.

Also we rotate and translate 
$\triangle$ in $\R^k$
so that $p_{\!_K}(s_{\circ})=f(x_{\circ})$ is the origin, and
the simplex $\tau$ of $\triangle$ whose relative interior contains the origin
is a coordinate plane whose dimension equals $\dim(\tau)$. 
Then write $p_{\!_K}(b, v)=(p_b(b), p_n(v))$ and $f(\b, \nu)=(f_b(\b), f_n(\nu))$
where $p_b$, $f_b$ are the base-coordinates, 
$p_n$, $f_n$ are the fiber-coordinates, and by construction
the bases and the fibers lie in coordinate planes. 
Recall that $p_b$ and $f_n$ are injective.

Group the coordinates as follows
\[\bm a=(v, \b),\ \
\bm b=(b, \nu),\ \ 
I(\bm b)=(p_{b}(b), f_{n}(\nu)), \ \ 
J(\bm a)=(f_b(\b), p_n(v))
\]  
so that $\G$ and $\mathcal P$ look in these coordinates in 
$\R^{k}\times\R^k\times\R^k\times\R^k$ as
\[
\G^\prime=\{(\bm a, \bm b, I(\bm b), J(\bm a))\,|\, (\bm a, \bm b)\in \Pi\}
\qquad
\mathcal P^\prime=\{(\bm a, \bm b, \bm u, \bm u)\,:\, 
\bm a, \bm b, \bm u\in \R^k\}
\]
The map $(x, y, z, w)\to (x, y, z+w, z-w)$, where $x, y, z, w\in\R^{k}$,
takes $\G^\prime$, $\mathcal P^\prime$ to
\[
\overline\G=\{(\bm a, \bm b, I(\bm b)+J(\bm a), I(\bm b)-J(\bm a))
\,|\, (\bm a, \bm b)\in \Pi\}
\]
\[
\overline{\mathcal P}=\{(\bm a, \bm b, \bm u, 0)\,|\, \bm a, \bm b, \bm u\in\R^k\}.
\]
The subspace
$Q=\{(\bm a, 0, 0, \bm v)\,|\, \bm a, 0, \bm v\in\R^k\}$ is orthogonal 
to $\overline{\mathcal P}$, and
$Q\cap \overline{\mathcal P}=\{(\bm a, 0, 0, 0)\,|\, \bm a\in\R^k\}$.
The injectivity of the map $\bm b\to I(\bm b)$ implies
that the orthogonal 
projection of $\overline{\G}$ onto $Q$ is one-to-one because 
from $\bm a$ and $I(\bm b)-J(\bm a)$ one recovers $I(\bm b)$, 
and hence $\bm b$. Thus $\overline{\G}$ is a graph of a map 
$h$ from a neighborhood $D$ of the origin in $Q$ to $\R^{2k}$,
and hence $\overline{\G}\cap\overline{\mathcal P}$ 
is a graph of $h$ over $D\cap \overline{\mathcal P}$.
Change the coordinates to $q=(\bm a, 0, 0, \bm v)$ and 
$p=(0, \bm b, \bm u, 0)$. 
The homeomorphism $(q, p)\to (q, q+h(p))$,  $q\in D$, $p\in \R^{2k}$
preserves $\overline{\mathcal P}$ and takes $D\times\{0\}$ to 
the graph of $h$. Hence the graph of $h$ 
is transverse to $\overline{\mathcal P}$.
Thus $\G$ and $\mathcal P$ are transverse.
\end{proof}

\section{Lifting and power maps}
\label{app: Lifting and power maps}  

This appendix presents an obstruction theory argument
that is used to prove the existence of smoothable thickenings with
prescribed rational characteristic classes. 
It is based on~\cite[Appendix B]{BelKap-jams} where
linear disk bundles were treated.

Let $K$ be the the weak product~\cite[p.28]{Whi78}
of the family $\{K(G_i, i)\}_{i\in I}$ 
of pointed Eilenberg-MacLane spaces such that each $G_i$ is
infinite cyclic. For $S\subset I$ let
\[K_S:=\{(x_i)\in K\,:\, x_i=\ast_i\ \ \text{if}\ \ i\notin S\},\]
where $\ast_i\in K(G_i, i)$ is the basepoint.
For an integer $n$ let $n_i$ denote 
a cellular basepoint-preserving self-map of $K(G_i, i)$ induced by
the ``multiplication by $n$'' on $G_i$, where
we let $n_i$ be the identity map if $n=1$.
The product of the maps $n_i$ over $i\in S$ 
is denoted by $\bm n\co K_S\to K_S$ and called 
an {\em $n$-power map of $K_S$}. 
Let $\overline{\bm n}\co K\to K$ be the map
that sends the coordinate $x_i$
of $(x_i)$ to $n_i(x_i)$ if $i\in S$ and to $x_i$ if
$i\in I\setminus S$. Thus $\overline{\bm n}\vert_{K_S}=\bm n$
and we call $\overline{\bm n}$ the {\em extension\,} of an 
$n$-power map of $K_S$.

Consider the following diagram of CW complexes and cellular maps
\begin{equation*}
\xymatrix@1{& B\ar[d]|{\overset{\,}{\b}} &\\
A_1\ar[ru]^{\a_1} & K 
& A_2\ar[lu]_{\a_2} \\
& X\ar[u]|{\overset{\,}{\overline{\bm n}\,\circ f}}\ar@{.>}[lu]^{f_{1}}\ar@{.>}[ru]_{f_{2}} &}
\end{equation*}
where $X$ is a finite complex $X$, 
$f\co X\to K$ is a cellular map, 
$K$ is $2$-connected, $A_1$, $B$, $A_2$ are simply-connected, 
and $\b$, $\a_1$, $\a_2$ are fibrations which are
rational homotopy equivalences.
It follows that the fibers of $\b$, $\b\circ\a_1$, $\b\circ\a_2$ 
are simply-connected spaces with finite homotopy groups.

Since $f(X)$ is compact, the definition of a weak product
places $f(X)$ inside some $K_J$ for a finite subset $J$ of $I$.
Hence $K_J$ is the product of finitely many $K(G_i, i)$-factors
indexed by $J$. In this setting we have the following:

\begin{thm} 
\label{thm: prescribe LR} 
There is a positive integer $n$ and maps $f_1$, $f_2$ such that 
the {\em extension\,} $\overline{\bm n}\co K\to K$ of an 
$n$-power map of $K_J$ makes the diagram commute. 
\end{thm}
\begin{proof}
First, we show that if $A$ is a finite abelian group,
then there is a power map 
that induces the zero map on $H^n(K_J; A)$ for every $n>0$.

We can assume that $A$ cyclic because 
the factorization of $A=\bigoplus_l \Z_{n_l}$ as a sum of 
finitely many cyclic groups
gives rise via long exact coefficient sequences to 
a natural factorization $H^n(K_J;A)=\bigoplus_l H^n(K_J; \Z_{n_l})$, and
if for every $l$ there is a power map that
annihilates $H^n(K_J; \Z_{n_l})$ for all $n>0$, then
their composite sends $H^{n}(K_J;A)$ to zero for all $n>0$.

Induct by the number of $K(\Z, i)$-factors of $K_J$.
If $K_J=K(\Z, i)$ and $A=\Z_{n_l}$, then an $\bm n_l$-power map
works~\cite[Lemma 4]{HanQua}.
For the induction step write $J$ as the disjoint union of the subsets 
$S, T$ so that $K_J=K_S\times K_T$.
By Hurewicz theorem $K_S$, $K_T$ have finitely generated homology groups, and
since $A$ is cyclic we have $A\otimes A\cong A$ and $\mathrm{Tor}(A, A)=0$,
which gives the natural
K\"unneth short exact sequence~\cite[Proposition VI.12.16]{Dol-book} 
{\small\begin{equation*}
0\to\!\!\bigoplus_{\substack{i+j=n}}\!\! H^i(K_S; A)\otimes 
H^j(K_T; A)\to H^n(K_J; A)\to\!\!\!\! 
\bigoplus_{\substack{i+j=n+1}}\!\!\! \mathrm{Tor}(H^i(K_S; A), H^j(K_T; A))\to 0.
\end{equation*}}
\vspace{-10pt}

By induction there are two power maps $\bm{s}$, $\bm{t}$
that respectively annihilate $H^i(K_S; A)$, $H^j(K_T; A)$ for any $i, j>0$, and
hence $\bm{st}$ sends all these groups to zero. Consider $\bm{st}$
as a self-map of $K_J$, and note that it sends to zero the kernel
and the quotient in the above K\"unneth sequence.
A diagram chase on the three copies of the K\"unneth sequence 
stacked on top of each other implies that the power map 
$\bm{st}\circ\bm{st}$ of $K_J$ 
annihilates $H^n(K_J; A)$.

Fix $r\in\{1,2\}$.
The obstruction classes to lifting the identity map
$\iota\co K_J\to K_J$ as a section of 
$\b\circ\a_r$ over $K_J$ lie in 
the groups $H^{j+1}(K_J, \pi_j(F_{_{\!\b\circ\a_r}}))$ 
where $F_g$ denotes the homotopy fiber of $g$.
Fix a power map $\bm k$ that annihilates this cohomology group
for all $j$.
Naturality of obstruction classes 
gives a map $\s_r\co K_J\to A_r$ such that $\beta\circ\a_r\circ\s_r$ is homotopic
to $\bm k=\iota\circ \bm k$. 

The cohomology classes of the difference cochains of $\a_1\circ\s_1$, $\a_1\circ\s_2$
lie in the groups $H^{j}(K_J, \pi_j(F_\b))$. 
Let $\bm l$ be a power map that annihilates this cohomology group
for all $j$.
By naturality of the difference cochains the maps
$\a_1\circ\s_1\circ\bm{l}$, $\a_2\circ\s_2\circ\bm{l}$ 
are homotopic. 
Setting $f_r:=\s_r\circ\bm{l}\circ f$ and $n:=kl$, we conclude that 
$\b\circ\a_r\circ f_r$ is homotopic to 
$\bm{k}\circ\bm{l}\circ f=\bm{n}\circ f=\overline{\bm{n}}\circ f$. 
\end{proof}

\begin{rmk}
\label{rmk: integer multiple}
The proof of Theorem~\ref{thm: prescribe LR}
implies that $n$ can be replaced by any positive integer multiple of $n$.
\end{rmk}

\begin{rmk}
Theorems~\ref{thm: realization} and~\ref{thm: realization codim 2} use the case 
when $\a_1=\a_2$ is the identity map of $B$ and $f_1=f_2$.
The full generality of the theorem is needed in Section~\ref{sec: Smoothable thickenings}.
\end{rmk}

\begin{rmk} In Theorem~\ref{thm: prescribe LR}, if
the homotopy class $[f]$ of $f$ is thought of as 
$(\a_i)$ in $\!\underset{i\in J}{\oplus} H^i(X)$ with $\a_i\in H^i(X)$,
then $[\bm n\circ f]$ corresponds to $n\a_i$. 
\end{rmk}

\section{Codimension two Poincar{\'e} embeddings}
\label{app: poincare emb}

In this appendix we review a result in~\cite{CS76} 
about compact manifolds
with two codimension $2$ spines one of which is locally flat. 

As usual, we work in the {\scshape pl} category.
Throughout this section $M$ be a closed oriented manifold and 
$W$ is a compact oriented manifold with boundary with $\dim(W)-\dim(M)=2$. 
For a vector bundle $\xi$ over $M$
let $p_\xi\co D_\xi\to M$ and $S_\xi=\d D_\xi$ denote 
the associated disk and sphere bundles, respectively.  

An oriented {\em h-Poincar{\'e} embedding\,} of $M$ into $W$ 
consists of an oriented $2$-plane bundle $\xi$ over $M$,
a finite Poincar{\'e} pair $(E, {S_\xi}\coprod F)$, and a homotopy equivalence 
$h\co (W, \d W)\to (E\cup_{S_\xi} D_\xi, F)$ that maps the fundamental class
$[W,\d W]$ to the class that corresponds
to the fundamental class of $(D_\xi, S_\xi)$
after excising $E\setminus S_\xi$.

Any embedding of $M$ into $W$ gives rise to a 
h-Poincar{\'e} embedding as follows, see~\cite[p.210]{CS76}. 
Let $\xi$ be the $2$-plane bundle over $M$ whose 
Euler class is the normal Euler class of $M$ in $W$.
Let $R$ be a regular neighborhood of $M$ in $W$; 
thus $C:=W\setminus\mathrm{Int}(R)$ deformation retracts onto $\d R$.
By~\cite[Proposition 1.6]{CS76} there is
a homology isomorphism $\check h\co (R,\d R)\to (D_\xi, S_\xi)$
such that $p_\xi\circ \check h\vert_M$ is homotopic to the identity of $M$.
Let $E$ be the quotient space of the disjoint union of
$C$ and $S_\xi$ by the equivalence
relation that identifies each
$x\in\d R$ with $\check h(x)\in S_\xi$, 
let $q\co C\coprod S_\xi\to E$
denote the corresponding quotient map,
and set $F=\d W$.
Since $C$ and $S_\xi$ are compact, 
$q\vert_C$ is a quotient map onto $E$, and 
$q\vert_{S_\xi}$ is a homeomorphism onto its image.
Then $q\vert_{\d R}$ can be identified with $\check h\vert_{\d R}$. 
Gluing $\check h\co R\to D_\xi$ and  
$q\vert_C\co C\to E$ along $\check h\vert_{\d R}$ gives 
the map $h$ as in the previous paragraph. 

The paper~\cite{CS76} studies when a Poincar{\'e} embedding is induced by
an embedding.

If $M\to W$ is a locally flat embedding, then its regular neighborhood
$R$ is a linear disk bundle~\cite[p.127]{Wal-book}, and the
corresponding h-Poincar{\'e} embedding can be described by
identifying $R$ with $D_\xi$, and $E$ with $W\setminus\mathrm{Int}(R)$,
and taking $h$ to be the identity map of $W$. 

A {\em map of h-Poincar{\'e} embeddings 
$\a$, $\b$ of $M$ into $W$\,} 
is a map of triples \[(E_\a, F_\a, D_{\xi_{\a}})\to 
(E_\b, F_\b, D_{\xi_{\b}})\]
such that $D_{\xi_{\a}}\to D_{\xi_{\b}}$ is a bundle map, and
the composite of $h_\a$ and
the map obtained by gluing $D_{\xi_{\a}}\to D_{\xi_{\b}}$ and 
$E_\a\to E_\b$ is homotopic to
$h_\b$ as a map of pairs with domain $(W,\d W)$.

We say that h-Poincar{\'e} embeddings $\a$, $\b$ of $M$ into $W$  are {\em isomorphic\,}
if there are maps $\a\to\b$ and $\b\to \a$ of h-Poincar{\'e} embeddings 
such that
both composites are homotopic to the identity through maps
of h-Poincar{\'e} embedding.

If $g\co M^\prime\to M$ is a homotopy equivalence of closed manifolds,
and $\a$ is an \mbox{h-Poincar{\'e}} embedding of $M$, then the
{\em pullback $g^*\!\a$ of $\a$}
is an h-Poincar{\'e} embedding of $M^\prime$ 
for which $\xi_{g^*\!\a}=g^*\xi_{\a}$,
$F_{g^*\!\a}=F_\a$, and $E_{g^*{\!\a}}$ is the union of $E_{\a}$
with the mapping cylinder of the bundle map 
$S_{\xi_{g^*{\!\a}}}\to S_{\xi_{\a}}$, and $h_{g^*{\!\a}}$
is the composite of $h_\a$ with a bundle map covering 
a homotopy inverse of $g$.

The following result is implicit in~\cite[Example 6.2.1]{CS76}.

\begin{lem}
Let $M_1$, $M_2$ be closed manifolds embedded into $W$, and
let $\a_1$, $\a_2$ be corresponding 
h-Poincar{\'e} embeddings. Suppose 
$W$ deformation retracts both onto $M_1$ and $M_2$.
If $M_2$ is locally flat, and the restriction of 
a deformation retraction $W\to M_2$ to $M_1$ is homotopic to
a homeomorphism $g\co M_1\to M_2$,  
then $g^*\!\a_2$ and $\a_1$ are isomorphic, or equivalently,
$\a_2$ is isomorphic to the pullback of $\a_1$ via $g^{-1}$.
\end{lem}
\begin{proof}
Label the objects related to the h-Poincar{\'e} embedding
of $M_i$ into $W$ with the subscript $i\in\{1,2\}$.
Set $B_i=\d C_i\setminus\d W$.
Let $r_{i}^t\co C_i\to B_i$ be a deformation retraction with
$t\in[0, 1]$, where $r_i^0$ is the identity map of $C_i$, 
and $r_{i}^1(C_i)=B_i$. 
Denote the mapping cylinder
of $r_{i}^1\vert_{\d W}$ by $T_i$. 
The map $(x,t)\to r_{i}^t(x)$, 
where $x\in\d W$, descends to a map 
$\r_i\co T_i\to C_i$.
Sending the equivalence class of 
$(x,t)$ to itself defines a map $\tau\co T_1\to T_2$.

Let $\bar g\co D_{\xi_1}\to D_{x_2}$ be the bundle map covering the
homeomorphism $g$. 
Let $S_1$ be the mapping cylinder of 
$h_1\co \d R_1\to S_{\xi_1}$.
Let $S_2$ be the mapping cylinder of 
$\bar g^{-1}\co S_{\xi_2}\to S_{\xi_1}$,
where $\d R_2=S_{\xi_2}$. Represent 
the equivalence class of a point in these mapping
cylinders as a pair: a point in $\d R_i$ and $s\in [0,1]$.
The map $(y, s)\to ({\bar g}(h_1(y)), s)$, where $y\in\d R_1$, descends to
a map $\s\co S_1\to S_2$, which is well-defined because
$(y,1)\sim h_1(y)=\bar g^{-1}({\bar g}(h_1(y)))\sim ({\bar g}(h_1(y)), 1)$. 

If we glue $S_i$ to $T_i$ along $\d R_i$, then the above map
$\r_i\co T_i\to C_i$ extends to a map $T_i\cup_{\d R_i}\! S_i\to E_i$
which is the identity on $S_{\xi_1}\coprod\d W$,
and is also a homotopy equivalence because both 
$T_i\cup_{\d R_i}\! S_i$ and $E_i$ deformation retract onto $S_{\xi_1}$.
Similarly, the map $T_1\cup_{\d R_1}\! S_1\to T_2\cup_{\d R_2}\! S_2$
glued from $\tau$ and $\s$ is a homotopy equivalence
that is the identity on $S_{\xi_1}\coprod\d W$. 
By~\cite[Proposition 0.19]{Hat-book} there are homotopy equivalences 
$T_i\cup_{\d R_i}\! S_i\to E_i$
and $T_1\cup_{\d R_1}\! S_1\to T_2\cup_{\d R_2}\! S_2$
rel $S_{\xi_1}\coprod\d W$. This yields a homotopy equivalence
$E_1\to E_2$ rel $S_{\xi_1}\coprod\d W$ which easily gives rise
to an isomorphism $g^*\!\a_2\cong\a_1$.
\end{proof}

\small
\bibliographystyle{amsalpha}
\bibliography{rn}

\providecommand{\bysame}{\leavevmode\hbox to3em{\hrulefill}\thinspace}
\providecommand{\MR}{\relax\ifhmode\unskip\space\fi MR }
\providecommand{\MRhref}[2]{%
  \href{http://www.ams.org/mathscinet-getitem?mr=#1}{#2}
}
\providecommand{\href}[2]{#2}
\begin{thebibliography}{BLW10}

\bibitem[AGG11]{AGG}
S.~Anan'in, C.~H. Grossi, and N.~Gusevskii, \emph{Complex hyperbolic structures
  on disc bundles over surfaces}, Int. Math. Res. Not. IMRN (2011), no.~19,
  4295--4375.

\bibitem[And87]{And-vb}
M.~T. Anderson, \emph{Metrics of negative curvature on vector bundles}, Proc.
  Amer. Math. Soc. \textbf{99} (1987), no.~2, 357--363.

\bibitem[AZ67]{AZ67}
M.~A. Armstrong and E.~C. Zeeman, \emph{Transversality for piecewise linear
  manifolds}, Topology \textbf{6} (1967), 433--466.

\bibitem[Bel97]{Bel-odd}
I.~Belegradek, \emph{Some curious {K}leinian groups and hyperbolic
  {$5$}-manifolds}, Transform. Groups \textbf{2} (1997), no.~1, 3--29.

\bibitem[BK03]{BelKap-jams}
I.~Belegradek and V.~Kapovitch, \emph{Obstructions to nonnegative curvature and
  rational homotopy theory}, J. Amer. Math. Soc. \textbf{16} (2003), no.~2,
  259--284, corrected in arxiv.org/abs/math/0007007.

\bibitem[BKS11]{BKS11}
I.~Belegradek, S.~Kwasik, and R.~Schultz, \emph{Moduli spaces of nonnegative
  sectional curvature and non-unique souls}, J. Differential Geom. \textbf{89}
  (2011), no.~1, 49--85.

\bibitem[BL12]{BL}
A.~Bartels and W.~L\"{u}ck, \emph{The {B}orel conjecture for hyperbolic and
  {${\rm CAT}(0)$}-groups}, Ann. of Math. (2) \textbf{175} (2012), no.~2,
  631--689.

\bibitem[BLW10]{BLW}
A.~Bartels, W.~L\"{u}ck, and S.~Weinberger, \emph{On hyperbolic groups with
  spheres as boundary}, J. Differential Geom. \textbf{86} (2010), no.~1, 1--16.

\bibitem[Bow95]{Bow95}
B.~H. Bowditch, \emph{Geometrical finiteness with variable negative curvature},
  Duke Math. J. \textbf{77} (1995), no.~1, 229--274.

\bibitem[Bow99]{Bow99}
\bysame, \emph{Convergence groups and configuration spaces}, Geometric group
  theory down under ({C}anberra, 1996), de Gruyter, Berlin, 1999, pp.~23--54.

\bibitem[Bre97]{Bre-sheaf}
G.~E. Bredon, \emph{Sheaf theory}, second ed., Graduate Texts in Mathematics,
  vol. 170, Springer-Verlag, New York, 1997.

\bibitem[Bro62]{Bro62}
M.~Brown, \emph{Locally flat imbeddings of topological manifolds}, Ann. of
  Math. (2) \textbf{75} (1962), 331--341.

\bibitem[Bro72]{Bro-book}
W.~Browder, \emph{Surgery on simply-connected manifolds}, Springer-Verlag, New
  York-Heidelberg, 1972, Ergebnisse der Mathematik und ihrer Grenzgebiete, Band
  65.

\bibitem[CD95]{ChaDav}
R.~M. Charney and M.~W. Davis, \emph{Strict hyperbolization}, Topology
  \textbf{34} (1995), no.~2, 329--350.

\bibitem[CS76]{CS76}
S.~E. Cappell and J.~L. Shaneson, \emph{Piecewise linear embeddings and their
  singularities}, Ann. of Math. (2) \textbf{103} (1976), no.~1, 163--228.

\bibitem[CS78]{CapSha-symp}
\bysame, \emph{An introduction to embeddings, immersions and singularities in
  codimension two}, Algebraic and geometric topology ({P}roc. {S}ympos. {P}ure
  {M}ath., {S}tanford {U}niv., {S}tanford, {C}alif., 1976), {P}art 2, Proc.
  Sympos. Pure Math., XXXII, Amer. Math. Soc., Providence, R.I., 1978,
  pp.~129--149.

\bibitem[DJ91]{DavJan}
M.~W. Davis and T.~Januszkiewicz, \emph{Hyperbolization of polyhedra}, J.
  Differential Geom. \textbf{34} (1991), no.~2, 347--388.

\bibitem[Dol72]{Dol-book}
A.~Dold, \emph{Lectures on algebraic topology}, Springer-Verlag, New
  York-Berlin, 1972, Die Grundlehren der mathematischen Wissenschaften, Band
  200.

\bibitem[DV09]{DV-book}
R.~J. Daverman and G.~A. Venema, \emph{Embeddings in manifolds}, Graduate
  Studies in Mathematics, vol. 106, American Mathematical Society, Providence,
  RI, 2009.

\bibitem[GKL01]{GKL}
W.~M. Goldman, M.~Kapovich, and B.~Leeb, \emph{Complex hyperbolic manifolds
  homotopy equivalent to a {R}iemann surface}, Comm. Anal. Geom. \textbf{9}
  (2001), no.~1, 61--95.

\bibitem[GLT88]{GLT}
M.~Gromov, H.~B. Lawson, Jr., and W.~Thurston, \emph{Hyperbolic {$4$}-manifolds
  and conformally flat {$3$}-manifolds}, Inst. Hautes \'{E}tudes Sci. Publ.
  Math. (1988), no.~68, 27--45 (1989).

\bibitem[Gro87]{Gro-hypgr}
M.~Gromov, \emph{Hyperbolic groups}, Essays in group theory, Math. Sci. Res.
  Inst. Publ., vol.~8, Springer, New York, 1987, pp.~75--263.

\bibitem[Hae61]{Hae61}
A.~Haefliger, \emph{Plongements diff\'{e}rentiables de vari\'{e}t\'{e}s dans
  vari\'{e}t\'{e}s}, Comment. Math. Helv. \textbf{36} (1961), 47--82.

\bibitem[Hat02]{Hat-book}
A.~Hatcher, \emph{Algebraic topology}, Cambridge University Press, Cambridge,
  2002.

\bibitem[HM74]{HM74}
M.~W. Hirsch and B.~Mazur, \emph{Smoothings of piecewise linear manifolds},
  Princeton University Press, Princeton, N. J.; University of Tokyo Press,
  Tokyo, 1974, Annals of Mathematics Studies, No. 80.

\bibitem[HQ18]{HanQua}
B.~Hanke and P.~Quast, \emph{{$\Gamma$}-structures and symmetric spaces},
  Algebr. Geom. Topol. \textbf{18} (2018), no.~2, 877--895.

\bibitem[Irw65]{Irw}
M.~C. Irwin, \emph{Embeddings of polyhedral manifolds}, Ann. of Math. (2)
  \textbf{82} (1965), 1--14.

\bibitem[Kap89]{Kap89}
M.~E. Kapovich, \emph{Flat conformal structures on three-dimensional manifolds:
  the existence problem. {I}}, Sibirsk. Mat. Zh. \textbf{30} (1989), no.~5,
  60--73, 216.

\bibitem[Kle]{Kle-MO}
J.~Klein, \emph{Characteristic classes for block bundles}, MathOverflow,
  https://mathoverflow.net/q/97484 (version: 2012-05-21).

\bibitem[Kui88]{Kui}
N.~H. Kuiper, \emph{Hyperbolic {$4$}-manifolds and tesselations}, Inst. Hautes
  \'{E}tudes Sci. Publ. Math. (1988), no.~68, 47--76 (1989).

\bibitem[LL19]{LevLid}
A.~S. Levine and T.~Lidman, \emph{Simply connected, spineless 4-manifolds},
  Forum Math. Sigma \textbf{7} (2019), 1--11.

\bibitem[LR65]{LR65}
R.~Lashof and M.~Rothenberg, \emph{Microbundles and smoothing}, Topology
  \textbf{3} (1965), 357--388.

\bibitem[Mat75]{Mat75}
Y.~Matsumoto, \emph{A {$4$}-manifold which admits no spine}, Bull. Amer. Math.
  Soc. \textbf{81} (1975), 467--470.

\bibitem[McC75]{Mc75}
C.~McCrory, \emph{Cone complexes and {PL} transversality}, Trans. Amer. Math.
  Soc. \textbf{207} (1975), 269--291.

\bibitem[McC77]{Mc-cone}
\bysame, \emph{Cone bundles}, Trans. Amer. Math. Soc. \textbf{228} (1977),
  157--163.

\bibitem[Mic64]{Mic64}
E.~Michael, \emph{Cuts}, Acta Math. \textbf{111} (1964), 1--36.

\bibitem[Mil64]{Mil-micro}
J.~Milnor, \emph{Microbundles. {I}}, Topology \textbf{3} (1964), no.~suppl,
  suppl. 1, 53--80.

\bibitem[Mil68]{Mil}
\bysame, \emph{On characteristic classes for spherical fibre spaces}, Comment.
  Math. Helv. \textbf{43} (1968), 51--77.

\bibitem[MM79]{MadMil-book}
I.~Madsen and R.~J. Milgram, \emph{The classifying spaces for surgery and
  cobordism of manifolds}, Annals of Mathematics Studies, vol.~92, Princeton
  University Press, Princeton, N.J.; University of Tokyo Press, Tokyo, 1979.

\bibitem[MV15]{MunVol-book}
B.~A. Munson and I.~Voli{\'{c}}, \emph{Cubical homotopy theory}, New
  Mathematical Monographs, vol.~25, Cambridge University Press, Cambridge,
  2015.

\bibitem[Ont20]{Ont-ihes}
P.~Ontaneda, \emph{Riemannian hyperbolization}, Publ. Math. Inst. Hautes
  \'{E}tudes Sci. \textbf{131} (2020), 1--72.

\bibitem[RS68a]{RS-I}
C.~P. Rourke and B.~J. Sanderson, \emph{Block bundles. {I}}, Ann. of Math. (2)
  \textbf{87} (1968), 1--28.

\bibitem[RS68b]{RS-II}
\bysame, \emph{Block bundles. {II}. {T}ransversality}, Ann. of Math. (2)
  \textbf{87} (1968), 256--278.

\bibitem[RS68c]{RS-III}
\bysame, \emph{Block bundles. {III}. {H}omotopy theory}, Ann. of Math. (2)
  \textbf{87} (1968), 431--483.

\bibitem[RS82]{RS-book}
\bysame, \emph{Introduction to piecewise-linear topology}, Springer Study
  Edition, Springer-Verlag, Berlin-New York, 1982, Reprint.

\bibitem[RW]{RW-MO}
O.~Randal-Williams, \emph{Whitney sum formula for topological {P}ontryagin
  classes}, MathOverflow, https://mathoverflow.net/q/361012 (version:
  2020-05-23).

\bibitem[Sie69]{Sie-open-collar}
L.~C. Siebenmann, \emph{On detecting open collars}, Trans. Amer. Math. Soc.
  \textbf{142} (1969), 201--227.

\bibitem[Spi67]{Spi67}
M.~Spivak, \emph{Spaces satisfying {P}oincar\'{e} duality}, Topology \textbf{6}
  (1967), 77--101.

\bibitem[Sta67]{Sta-tata}
J.~R. Stallings, \emph{Lectures on polyhedral topology}, Notes by G. Ananda
  Swarup. Tata Institute of Fundamental Research Lectures on Mathematics, No.
  43, Tata Institute of Fundamental Research, Bombay, 1967.

\bibitem[Sto72]{Sto72}
D.~A. Stone, \emph{Stratified polyhedra}, Lecture Notes in Mathematics, Vol.
  252, Springer-Verlag, Berlin-New York, 1972.

\bibitem[Ven98]{Ven}
G.~A. Venema, \emph{A manifold that does not contain a compact core}, Topology
  Appl. \textbf{90} (1998), no.~1-3, 197--210.

\bibitem[Wal99]{Wal-book}
C.~T.~C. Wall, \emph{Surgery on compact manifolds}, second ed., Mathematical
  Surveys and Monographs, vol.~69, American Mathematical Society, Providence,
  RI, 1999, Edited and with a foreword by A. A. Ranicki.

\bibitem[Whi78]{Whi78}
G.~W. Whitehead, \emph{Elements of homotopy theory}, Graduate Texts in
  Mathematics, vol.~61, Springer-Verlag, New York-Berlin, 1978.

\end{thebibliography}

\end{document}